\documentclass[leqno,10pt]{amsart}
\usepackage{amsmath,amstext,amssymb,amsopn,amsthm,wasysym,mathrsfs}

\theoremstyle{plain}

\allowdisplaybreaks

\newtheorem{thm}{Theorem}[section]
\newtheorem{cor}[thm]{Corollary}
\newtheorem{pro}[thm]{Proposition}
\newtheorem{ob}[thm]{Observation}
\newtheorem{lem}[thm]{Lemma}
\newtheorem{rem}[thm]{Remark}
\newtheorem*{ex*}{Example}

\newcommand{\N}{\mathbb{N}}

\def\P{\mathcal P}
\def\m{\mu}
\def\J{\mathcal J}

\def\eps{\varepsilon}

\def\a{\alpha}
\def\b{\beta}
\def\ab{\alpha,\beta}
\def\t{\theta}
\def\vp{\varphi}
\def\st{\sin \frac{\theta}{2}}
\def\svp{\sin \frac{\varphi}{2}}
\def\ct{\cos \frac{\theta}{2}}
\def\cvp{\cos \frac{\varphi}{2}}
\def\sht{\sinh \frac{t}{2}}
\def\cht{\cosh \frac{t}{2}}
\def\Puv{\Psi^{\alpha,\beta}(t,\theta,\varphi,u,v)}

\def\pia{d\Pi_{\alpha}(u)}
\def\pib{d\Pi_{\beta}(v)}
\def\piaa{d\Pi_{\alpha+1}(u)}
\def\pibb{d\Pi_{\beta+1}(v)}
\def\piu{d\Pi_{-1\slash 2}(u)}
\def\piv{d\Pi_{-1\slash 2}(v)}

\def\oa{\Pi_{\alpha}(u)\,du}
\def\ob{\Pi_{\beta}(v)\,dv}

\def\q{\mathfrak{q}}
\def\v{\varphi}

\DeclareMathOperator{\support}{supp}
\DeclareMathOperator{\spann}{span}
\DeclareMathOperator{\spec}{spec}
\DeclareMathOperator{\pv}{P.V.}
\DeclareMathOperator{\id}{Id}
\DeclareMathOperator{\sgn}{sgn}
\DeclareMathOperator*{\essup}{ess\,sup}

\def\Ptilde{\widetilde{\Psi}^{\lambda}(t,\mathfrak{q})}

\def\piK{d\Pi_{\alpha, K}(u)}
\def\piR{d\Pi_{\beta, R}(v)}
\def\abpiK{\Pi_{\alpha, K}}

\def\lam{\lambda}

\def\Leb1{L^1(d\mu_{\alpha,\beta})}
\def\Lebr1{L^1((0,\pi/4),d\mu_{\alpha,\beta})}

\title[{Riesz-Jacobi transforms as $\pv$ integrals}]
	{Riesz-Jacobi transforms as principal value integrals}

\author[A.J. Castro]{Alejandro J. Castro}
\address{Alejandro J. Castro \newline
		Departamento de An\'alisis Matem\'atico, 
		Universidad de La Laguna \newline
		Campus de Anchieta, Avda. Astrof\'{\i}sico Francisco S\'anchez, s/n, \newline
		38271 La Laguna (Sta. Cruz de Tenerife), Spain
		}
\email{ajcastro@ull.es}

\author[A. Nowak]{Adam Nowak}
\address{Adam Nowak, \newline
           Institute of Mathematics,
       Polish Academy of Sciences, \newline
      \'Sniadeckich 8,
      00--656 Warszawa, Poland
      }
\email{anowak@impan.pl}

\author[T.Z. Szarek]{Tomasz Z. Szarek}
\address{Tomasz Z. Szarek,     \newline
           Institute of Mathematics,
       Polish Academy of Sciences, \newline
      \'Sniadeckich 8,
      00--656 Warszawa, Poland
      }
\email{szarektomaszz@gmail.com}

\begin{document}

\begin{abstract}
We establish an integral representation for the Riesz transforms naturally 
associated with classical Jacobi expansions.
We prove that the Riesz-Jacobi transforms of odd orders express as principal value integrals against
kernels having non-integrable singularities on the diagonal. On the other hand, we show that the Riesz-Jacobi
transforms of even orders are not singular operators. In fact they are given as usual integrals against
integrable kernels plus or minus, depending on the order, the identity operator.
Our analysis indicates that similar results, existing in the literature and corresponding to several
other settings related to classical discrete and continuous orthogonal expansions, should be reinvestigated
so as to be refined and in some cases also corrected.
\end{abstract}

\maketitle

\footnotetext{
\indent	
The first-named author was partially supported by MTM2010/17974 and also by an FPU grant 
from the Government of Spain.
Research of the second-named author was partially supported by the National Science Centre of Poland,
project no.\ 2013/09/B/ST1/02057.
The third-named author was partially supported by the National Science Centre of Poland, 
project no.\ 2012/05/N/ST1/02746.
}

\section{Introduction} \label{sec:intro}

The classical Riesz transforms in $\mathbb{R}^n$, $n \ge 1$, are formally given by
$$
R_j = \partial_j (-\Delta)^{-1/2}, \qquad j =1,\ldots,n.
$$
These identities have a strict meaning when understood in the sense of the Fourier transform and thus define
the Fourier multipliers
$$
\widehat{R_jf}(\xi) = i \frac{\xi_j}{|\xi|} \hat{f}(\xi), \qquad \textrm{where} \quad 
	\hat{f}(\xi) = \int_{\mathbb{R}^n} f(x) \, e^{-2\pi i \langle x, \xi\rangle} \, dx.
$$
It is well known that the $R_j$, $j=1,\ldots,n$, possess the singular integral representation
$$
R_jf(x) = \frac{\Gamma(\frac{n+1}2)}{\pi^{\frac{n+1}2}} 
	\pv \int_{\mathbb{R}^n} \frac{y_j-x_j}{|y-x|^{n+1}} f(y) \, dy
$$
in $L^p(\mathbb{R}^n)$, $1 \le p < \infty$. The last integral does not make the usual sense for 
$x \in \support f$, because of the non-integrable kernel singularity along the diagonal. 
But it exists in the principal value sense thanks to subtle cancellations around $y=x$.
An important special case is $n=1$ and the Hilbert transform
$$
Hf(x) = \frac{1}{\pi} \pv \int_{-\infty}^{\infty} \frac{f(y)}{y-x} \, dy.
$$
Note that
\begin{equation} \label{hid}
H^{2k} = (-1)^k \id, \qquad H^{2k-1} = (-1)^{k+1}H, \qquad k \ge 1.
\end{equation}
Another classical example of a singular integral operator is the conjugate function mapping
on the torus,
$$
\mathcal{C} \colon f \mapsto \sum_{k \in \mathbb{Z}} i \sgn(k) \hat{f}(k)\, e^{2\pi i k x},
\qquad \textrm{where} \quad \hat{f}(k) = \int_0^1 f(x) e^{-2\pi i k x}\, dx.
$$
On the subspace of $L^2(0,1)$ of functions having vanishing mean value, this can be written in a compact way
as $\mathcal{C} = \frac{d}{dx}(-\Delta)^{-1/2}$ and it is easy to check that 
identities analogous to \eqref{hid} hold. Moreover, in $L^1(0,1)$ we have the integral representation
(cf.\ \cite[Chapter VII]{Z})
$$
\mathcal{C}f(x) = \pv \int_0^1 \cot\big( \pi (y-x)\big) f(y)\, dy.
$$
All the above mentioned operators were intensively studied in the first half of the 20th century.
Then their analogues were defined and investigated in a great variety of contexts.

Riesz transforms related to classical discrete and continuous orthogonal expansions are defined according to
the following general scheme, cf.\ \cite{NoSt0}. For the sake of clarity we restrict here
to one-dimensional discrete expansions. Let $\{\phi_n\}$ be an orthogonal basis in $L^2((a,b),d\mu)$,
$- \infty \le a < b \le \infty$, consisting of eigenfunctions of a `Laplacian' $L$. Typically, 
and thus also here,
$L$ is a symmetric and non-negative in $L^2(d\mu)$ second order differential operator. We assume that
$L$ can be decomposed as $L = \delta^*\delta + c$, where $c\ge 0$ is a constant, $\delta$ is a first
order differential operator, and $\delta^*$ is its formal adjoint in $L^2(d\mu)$. Then $\delta$
is a natural derivative associated with $L$. In these circumstances, the Riesz transform of arbitrary order
$N \ge 1$ is formally defined as
$$
R_N = \delta^N L^{-N/2}.
$$
This identity can be understood strictly in the spectral sense, in some cases after restricting to a
suitable subspace of $L^2(d\mu)$. Usually, it is not hard to associate with $R_N$ an integral kernel
$R_N(x,y)$ so that
$$
R_Nf(x) = \int R_{N}(x,y) f(y)\, d\mu(y), \qquad x \notin \support f,
$$
for suitable $f$. However, the question of deriving an integral representation valid also 
on the support of $f$ is a subtle and complicated matter. Indeed, comparing to the classical case, here
$R_N(x,y)$ is in general a non-convolution kernel expressed only implicitly, often via integrals involving
transcendental special functions or oscillating series. 
Furthermore, typically the `derivative' $\delta$ is not
skew-adjoint and does not commute with $L$, so no direct analogues of \eqref{hid} can be hoped for.
Consequently, higher order Riesz transforms require a distinct analysis.

We now give a heuristic description of the approach proposed in this paper. 
We believe that it is of independent
interest since it applies to a quite general situation covering a number of settings where similar questions
were investigated earlier, as commented in more detail below. Taking into account the decomposition
$L = \delta^{*} \delta + c$ and the fact that $\delta^* = -\delta + R(0)$, 
we infer that $\delta^{2k}$, $k \ge 1$, can be written as
$$
\delta^{2k} = (-1)^k L^{k} + R(2k-1),
$$
where $R(m)$ stands for a generic differential operator of order $m$. Then, formally,
\begin{align*}
R_{2k} & = \delta^{2k} L^{-k} = (-1)^k \id + R(2k-1) L^{-k}, \\
R_{2k+1} & = \delta \delta^{2k} L^{-k} L^{-1/2} = (-1)^k R_1 + R(2k) L^{-k-1/2}.
\end{align*}
These counterparts of \eqref{hid} 
can easily be given a strict meaning on $\spann\{\phi_n\}$. Now, since the order of
$R(2k-1)$ is smaller than that of $L^k$, the operator $R(2k-1)L^{-k}$ is not singular 
in the sense that it corresponds to an integrable kernel and hence should admit
a usual integral representation, and the same for $R(2k)L^{-k-1/2}$. Moreover, the main singularity
of $R_{2k+1}$ is carried by $R_1$, so the study of singular integral representation for $R_N$ is
reduced to the analogous problem for $R_1$. In view of the above, we postulate the following representation
for sufficiently regular $f$:
\begin{align}
R_Nf(x) & = \pv \int R_N(x,y) f(y)\, d\mu(y), \qquad N \;\; \textrm{odd}, \label{Ro} \\
R_Nf(x) & = (-1)^{N/2}f(x) + \int R_N(x,y) f(y)\, d\mu(y), \qquad N \;\; \textrm{even}. \label{Re}
\end{align}
The main objective of this paper is to prove \eqref{Ro} and \eqref{Re}
in the context of classical Jacobi expansions,
see Theorems \ref{thm:main} and \ref{cor:mainmain} in Section \ref{sec:prel}. 
It is remarkable that no $\pv$ is needed to represent
the Riesz transforms of even orders. On the other hand, $\pv$ is absolutely essential in case of odd orders.

We claim that \eqref{Ro} and \eqref{Re} are true in many other particular contexts, including expansions into
Hermite and Laguerre polynomials/functions, and continuous Fourier-Bessel expansions
(contexts of modified and non-modified Hankel transforms). Unfortunately, we are not able to prove
this claim here, since this would in fact require writing a separate paper(s). 
Integral representations of Riesz
transforms in the just mentioned settings can be found in the literature, though not in an optimal form
and not always correct form. To give some concrete examples, let us focus first on the context of Hermite
polynomial expansions (in this situation $L$ is the Ornstein-Uhlenbeck operator). In \cite{U} the Riesz
operators are represented for all $x$ as $R_Nf(x) = \int R_N(x,y) f(y) d\mu(y)$, $N \ge 1$, 
with no $\pv$ involved.
An improved expression $R_Nf(x) = \pv \int R_N(x,y) f(y) d\mu(y)$, $N \ge 1$, can be found in
\cite{FS,P,PS}. But this is still inaccurate, and the same problem recurs in the study of a variant
of $R_N$ \cite{AFS}, in a more general Hermite framework \cite{FHS}, as well as in the setting related to
Laguerre polynomial expansions \cite{FSS}. The correct representation
\begin{equation} \label{Rcor}
R_Nf(x) = a_N f(x) + \pv \int R_N(x,y) f(y) d\mu(y), \qquad N \ge 1,
\end{equation}
for the Riesz-Hermite transforms appears first in \cite{GMST}, though with
unspecified coefficients $a_N$; see also \cite[Section 6]{S}.
In the context of continuous Fourier-Bessel expansions, as well as in the
contexts of Hermite and Laguerre function expansions, the representation \eqref{Rcor} was established
in \cite{BFMR} and \cite{BFRS}, respectively, but with miscalculated coefficients $a_N$. Finally,
\eqref{Rcor} with explicit and correct $a_N$ was derived recently in \cite{BFRT} in the ultraspherical
setting. The latter is a special case of the Jacobi framework investigated in this paper. Apparently,
the fact that the $\pv$ is superfluous in case of even orders was overlooked in the literature.

Our strategy of proving \eqref{Ro} and \eqref{Re} in the Jacobi setting is simpler than that elaborated
in \cite{BFMR,BFRS,BFRT}. Roughly, the main differences are that here we reduce the problem to showing
\eqref{Ro} with $N=1$ and then we verify the principal value integral representation directly,
not via comparing with some other, already known situation. 
Noteworthy, our methods involve very recent techniques and results obtained in the Jacobi setting, see
\cite{L,NR,NoSj,NoSj2,NSS}. On the other hand,
they require elaborating some new technical tools that may be useful elsewhere.
For instance, in Lemma \ref{lem:thetaJP} we obtain quite precise estimates of derivatives of the
Jacobi-Poisson kernel.

Recently, an alternative notion of higher order Riesz transform $\mathcal{R}_N$ associated with $L$ was
proposed in \cite{NoSt}. In some aspects, $\mathcal{R}_N$ seems more natural than $R_N$.
Accordingly, in this paper we study also $\mathcal{R}_n$ in the Jacobi setting and
establish its singular integral representation, which occurs to be analogous to \eqref{Ro} and \eqref{Re}; 
see Theorem \ref{thm:main2}. We take this opportunity to show that $\mathcal{R}_N$ has in general better
mapping properties than $R_N$, see Remark \ref{rem:bndRe} and Proposition \ref{pro:unbound}. 
The latter is a supplementary significant 
result of this paper, which reveals a new and interesting phenomenon.

The paper is organized as follows. In Section \ref{sec:prel} we introduce the Jacobi setting and state
the main results, that is Theorems \ref{thm:main}, \ref{cor:mainmain}, \ref{thm:main2} and Proposition
\ref{pro:unbound}. Section \ref{sec:proof} is devoted to the proofs of Theorems \ref{thm:main},
\ref{cor:mainmain} and \ref{thm:main2}. Some technical results needed in this section are proved
in the subsequent Sections \ref{sec:thetaJP}-\ref{sec:PV0}. 
Appendix contains the proof of Proposition \ref{pro:unbound}.

\subsection*{Notation}
Throughout the paper we use a standard notation consistent with that used in \cite{NoSj,NSS};
we refer there for any unexplained notation or symbols. In what follows, all the principal value
integrals over $(0,\pi)$ are understood according to the equivalence
$$
\pv \int_0^{\pi} K(\t,\v) f(\v)\, d\mu(\v) \equiv 
\lim_{\eps \to 0^+} \int_{0 < \v < \pi,\, |\v-\t|>\eps} K(\t,\v)f(\v)\, d\mu(\v) .
$$
When writing estimates, we will frequently use the notation $X \lesssim Y$ to indicate that
$X \le C Y$ with a positive constant $C$ independent of significant quantities. We shall write
$X \simeq Y$ when simultaneously $X \lesssim Y$ and $Y \lesssim X$.

\section{Preliminaries and statement of results} \label{sec:prel}

As in \cite{NoSj,NoSj2,NSS},
we consider the setting related to expansions into Jacobi trigonometric polynomials.
Let $\ab > -1$. The normalized trigonometric Jacobi polynomials are given by
$$
\mathcal{P}_n^{\ab}(\t) = c_n^{\ab} P_n^{\ab}(\cos\t), \qquad \t \in (0,\pi),
$$
where $c_n^{\ab}$ are normalizing constants, and $P_n^{\ab}$, $n\ge 0$, are the classical
Jacobi polynomials as defined in Szeg\H o's monograph \cite{Sz}. 
The system $\{\mathcal{P}_n^{\ab}: n \ge 0\}$ is an orthonormal basis in $L^2(d\mu_{\ab})$,
where $\mu_{\ab}$ is a measure on the interval $(0,\pi)$ defined by
$$
d\mu_{\ab}(\t) = \Big( \sin\frac{\t}2 \Big)^{2\alpha+1} \Big( \cos\frac{\t}2\Big)^{2\beta+1} d\t.
$$
It consists of eigenfunctions of the Jacobi differential operator
$$
\mathcal{J}^{\ab} = - \frac{d^2}{d\theta^2} - \frac{\alpha-\beta+(\alpha+\beta+1)\cos\theta}{\sin \theta}
    \frac{d}{d\theta} + \tau_{\ab}^2, \qquad \textrm{where} \qquad 
    \tau_{\ab} =  \frac{\alpha+\beta+1}{2}
$$
(notice that $\tau_{\ab}$ may be negative); more precisely,
$$
\mathcal{J}^{\ab} \mathcal{P}_n^{\ab} = \big( n + \tau_{\ab}\big)^2 \mathcal{P}_n^{\ab},
    \qquad n \ge 0.
$$
We shall denote by the same symbol $\mathcal{J}^{\ab}$ the natural self-adjoint extension
in $L^2(d\mu_{\ab})$ whose
spectral resolution is given by the $\mathcal{P}_n^{\ab}$, see \cite[Section 2]{NoSj} for details.
Further, by $H_t^{\ab}(\t,\v)$ we will denote the integral kernel of the Jacobi-Poisson semigroup
$\{\exp(-t(\mathcal{J}^{\ab})^{1/2})\}$,
\begin{equation} \label{d0}
H_t^{\ab}(\t,\v) =
    \sum_{n=0}^{\infty} \exp\big(-t | n + \tau_{\ab}| \big)\P_n^{\ab}(\t) \P_n^{\ab}(\v),
		\qquad t>0, \quad \t,\v \in (0,\pi).
\end{equation}
The last series can be repeatedly differentiated term by term in $t$, $\t$ and $\v$, and hence defines
a smooth function of $(t,\t,\v) \in (0,\infty)\times (0,\pi)\times (0,\pi)$.
There is no satisfactory explicit expression for $H_t^{\ab}(\t,\v)$. Nonetheless, sharp estimates of
this kernel were found recently in \cite[Section 6]{NSS}, see also \cite[Appendix]{NoSj2}. 
Note that for the special choice $\alpha=\beta=\lambda-1/2$ the whole situation becomes the ultraspherical
setting with parameter $\lambda$ investigated in \cite{BFRT} and many other papers.

The Riesz-Jacobi transform $R_N^{\ab}$ of order $N \ge 1$ is formally defined by (cf. \cite{NoSj,NSS})
\begin{equation} \label{d1}
R_N^{\ab} \colon f \mapsto \partial^N \big(\mathcal{J}^{\ab}\big)^{-N/2}f.
\end{equation}
Here $\partial$ is the usual derivative, and its relevance is motivated by the factorization
$$
\J^{\ab} = \delta^*\delta + \tau_{\ab}^2, 
$$
where $\delta= \partial$ and
$\delta^{*} = -\delta - (\alpha+1/2)\cot\frac{\t}2+(\beta+1/2)\tan\frac{\t}2$
is the formal adjoint of $\delta$ in $L^2(d\m_{\ab})$.
It is known that replacing $\partial = \delta$ by $\delta^*$ in \eqref{d1} 
is not appropriate since even for $N=1$ 
this would lead to operators mapping outside $L^2(d\mu_{\ab})$; see \cite[Remark 2.6]{NoSj}.

We now focus on understanding \eqref{d1} in a strict way. For $f \in L^2(d\mu_{\ab})$ the negative
power of $\mathcal{J}^{\ab}$ is naturally given by the $L^2(d\mu_{\ab})$-convergent spectral series
\begin{equation} \label{d2}
\big(\mathcal{J}^{\ab}\big)^{-N/2}f = \sum_{n=0}^{\infty} \big|n+\tau_{\ab}\big|^{-N}
    \big\langle f,\P_n^{\ab}\big\rangle_{d\mu_{\ab}} \P_n^{\ab},
\end{equation}
provided that $\tau_{\ab} \neq 0$; otherwise the bottom eigenvalue of $\mathcal{J}^{\ab}$ is $0$ and
\eqref{d2} 
does not make sense. 
For $f \in \spann\{\P_n^{\ab}:n \ge 0\}$
the series terminates and so the sum is in fact finite. In this way \eqref{d1} defines strictly
$R_N^{\ab}$ ($R_N^{\ab}f$ is defined pointwise)
on the dense subspace $\spann\{\P_n^{\ab}:n \ge 0\}$ of $L^2(d\mu_{\ab})$, if only
$\tau_{\ab} \neq 0$. Similarly, when $\tau_{\ab}=0$, \eqref{d1} defines strictly $R_N^{\ab}$
on the subspace $\spann\{\P_n^{\ab}:n \ge 1\}$ of codimension $1$, which is dense in
$\{\P_0^{\ab}\}^{\perp}\subset L^2(d\mu_{\ab})$. To treat uniformly all $\ab > -1$, it is reasonable to
make the convention that, in case $0 \in \spec\mathcal{J}^{\ab}$ (i.e.\ $\tau_{\ab}=0$), before
applying $\big(\mathcal{J}^{\ab}\big)^{-N/2}$ in \eqref{d1} $f$ is projected orthogonally onto
$\{\P_0^{\ab}\}^{\perp}$; for further reference call this projection $\Pi_{0}$.
With this convention, \eqref{d1} defines pointwise $R_N^{\ab}f$ for
$f \in \spann\{\P_n^{\ab}:n \ge 0\}$ and all $\ab >-1$.
Notice that since $\P_0^{\ab}$ is a constant function, actually for all $\ab > -1$ we have
$$
R_N^{\ab}f(\t) = \partial^N \big(\J^{\ab}\big)^{-N/2}\Pi_0 f(\t), 
	\qquad \t \in (0,\pi), \quad f \in \spann\{\P_n^{\ab} : n \ge 0\}.
$$

As was shown in \cite[Section 3]{NoSj}, $R_N^{\ab}$ extends uniquely from $\spann\{\P_n^{\ab}:n \ge 0\}$
to a bounded linear operator on $L^2(d\mu_{\ab})$ given by
\begin{equation} \label{d3}
R_N^{\ab}f = \sum_{n=1}^{\infty} \big|n+\tau_{\ab}\big|^{-N}
    \big\langle f,\P_n^{\ab}\big\rangle_{d\mu_{\ab}} \partial^N\P_n^{\ab},
\end{equation}
the series being convergent in $L^2(d\mu_{\ab})$. Moreover, according to \cite{NoSj,NSS}, $R_N^{\ab}$ is a
Calder\'on-Zygmund operator in the sense of the space of homogeneous type $((0,\pi),d\mu_{\ab},|\cdot|)$.
In particular, it extends uniquely from $L^2(d\m_{\ab})\cap L^p(wd\mu_{\ab})$ to a bounded operator on
$L^p(wd\mu_{\ab})$, $w\in A_p^{\ab}$, $1<p<\infty$, and, when $p=1$,
to a bounded operator from $L^1(wd\mu_{\ab})$ to weak $L^1(wd\mu_{\ab})$, $w \in A_1^{\ab}$.
Here $A_p^{\ab}$ stands for the Muckenhoupt class of weights related to our space
of homogeneous type (see \cite[Section~1]{NoSj} for the definition).
From the above cited papers we also know that $R_N^{\ab}$ is associated with the kernel
\begin{equation} \label{d4}
R_N^{\ab}(\t,\v) = \frac{1}{\Gamma(N)} \int_0^{\infty} \partial_{\t}^{N} H_t^{\ab}(\t,\v) t^{N-1}\, dt,
    \qquad \t,\v \in (0,\pi), \quad \t \neq \v
\end{equation}
(the integral here converges absolutely when $\t\neq \v$) in the sense that for, say, 
$f \in L^{\infty}(0,\pi)$
\begin{equation} \label{d44}
R_N^{\ab}f(\t) = \int_0^{\pi} R_N^{\ab}(\t,\v)f(\v)\, d\mu_{\ab}(\v), \qquad \textrm{a.a.}\;\;
    \t \notin \support f.
\end{equation}
Recall that this kind of association identifies a Calder\'on-Zygmund operator up to a pointwise
multiplication operator. Notice that the integral in \eqref{d44} diverges for $\t \in \support f$
when there is a non-integrable singularity of the kernel $R_N^{\ab}(\t,\v)$ along the diagonal.

Next, we focus on representing pointwise $R_N^{\ab}f$ for sufficiently regular functions,
say $f \in C_c^{\infty}(0,\pi)$
(smooth functions with $\support f \subset (0,\pi)$). Such functions belong to the domain
of $\mathcal{J}^{\ab}$ and are dense in $L^p(wd\mu_{\ab})$,
$w \in A_p^{\ab}$, $1 \le p < \infty$. We will need the following.
\begin{pro} \label{prop:FJdecay}
Let $\ab>-1$ and $f \in C_c^{\infty}(0,\pi)$ be fixed. Given any $M \in \N$, we have
$$
\big| \langle f, \P_n^{\ab}\rangle_{d\mu_{\ab}}\big| \lesssim (n+1)^{-2M}, \qquad n \ge 0.
$$
\end{pro}

\begin{proof}
Using the symmetry of $\mathcal{J}^{\ab}$ and the Schwarz inequality we can write
\begin{align*}
\big|n+\tau_{\ab}\big|^{2M} \big|\big\langle f,\P_n^{\ab}\big\rangle_{d\mu_{\ab}}\big| &
= \big| \big\langle f, (\mathcal{J}^{\ab})^M \P_n^{\ab}\big\rangle_{d\mu_{\ab}}\big|
 = \big|\big\langle (\mathcal{J}^{\ab})^{M}f,\P_n^{\ab}\big\rangle_{d\mu_{\ab}}\big| \\
& \le \big\| (\mathcal{J}^{\ab})^M f\big\|_{L^2(d\mu_{\ab})}.
\end{align*}
The conclusion follows.
\end{proof}
Further, for each $N \ge 0$ we note the estimate
\begin{equation} \label{d5}
\big|\partial_{\t}^{N}\P_n^{\ab}(\t)\big| \lesssim (n+1)^{c}, \qquad \t \in (0,\pi), \quad n \ge 0,
\end{equation}
where $c=c(\alpha,\beta,N)$ can be taken as $\alpha+\beta+2+3N$. This follows from the bound 
(see \cite[Section 2]{NoSj})
\begin{equation} \label{d55}
|\P_n^{\ab}(\t)| \lesssim (n+1)^{\alpha+\beta+2}, \qquad \t \in (0,\pi), \quad n \ge 0,
\end{equation}
and the differentiation rule
\begin{equation} \label{derP}
\partial_{\t} \P_n^{\ab}(\t)
    = -\frac{1}2 \sqrt{n(n+\alpha+\beta+1)} \, \sin\t \,\P_{n-1}^{\alpha+1,\beta+1}(\t), \qquad n \ge 1.
\end{equation}

Combining Proposition \ref{prop:FJdecay} and \eqref{d5} we see that for $f \in C_c^{\infty}(0,\pi)$
the series defining $(\mathcal{J}^{\ab})^{-N/2}f$ in \eqref{d2} converges pointwise and uniformly in
$\t \in (0,\pi)$
(with suitable modification in case $\tau_{\ab}=0$, according to our convention). 
Moreover, term by term differentiation shows that
$(\mathcal{J}^{\ab})^{-N/2}f \in C^{\infty}(0,\pi)$. Similarly, still for $f\in C_c^{\infty}(0,\pi)$,
the series defining $R_N^{\ab}f$ in \eqref{d3} converges pointwise and uniformly in $\t\in (0,\pi)$,
and $R_N^{\ab}f \in C^{\infty}(0,\pi)$. Notice that one can exchange the order of summation and
differentiation in \eqref{d3} if $f \in C_c^{\infty}(0,\pi)$. In particular,
\begin{equation} \label{d6}
R_N^{\ab}f(\t) = \partial_{\t}^N (\mathcal{J}^{\ab})^{-N/2}f(\t),
    \qquad \t \in (0,\pi), \quad f \in C_c^{\infty}(0,\pi),
\end{equation}
where $f$ on the right-hand side must be replaced by $\Pi_0 f$ in case $\tau_{\ab}=0$ (one can of course do the replacement
also in the opposite case).

In \eqref{d6} one can write $(\mathcal{J}^{\ab})^{-N/2}$
(or $(\mathcal{J}^{\ab})^{-N/2}\Pi_0$) as an integral against the
potential kernel (compensated potential kernel).
Indeed, let us first assume that $\tau_{\ab}\neq 0$. With the aid of Proposition
\ref{prop:FJdecay}, \eqref{d55} and Fubini's theorem we see that
\begin{align} \nonumber
\big(\mathcal{J}^{\ab}\big)^{-N/2}f(\t) & = \sum_{n=0}^{\infty} \big|n+\tau_{\ab}\big|^{-N}
    \langle f,\P_n^{\ab}\rangle_{d\mu_{\ab}} \P_n^{\ab}(\t) \\ \nonumber
& = \sum_{n=0}^{\infty} \bigg[ \frac{1}{\Gamma(N)} \int_0^{\infty}
 \exp\big(-t|n+\tau_{\ab}|
    \big) t^{N-1}\, dt \bigg] \, \langle f,\P_n^{\ab}\rangle_{d\mu_{\ab}} \P_n^{\ab}(\t) \\ \nonumber
& = \frac{1}{\Gamma(N)} \int_0^{\infty} \sum_{n=0}^{\infty} \exp\big(-t|n+\tau_{\ab}|
    \big) \langle f,\P_n^{\ab}\rangle_{d\mu_{\ab}} \P_n^{\ab}(\t) \, t^{N-1}\, dt \\ \label{ch1}
& = \frac{1}{\Gamma(N)} \int_0^{\infty} \int_0^{\pi} H_t^{\ab}(\t,\v) f(\v)\, d\mu_{\ab}(\v)\, t^{N-1}\, dt \\ \nonumber
& = \int_0^{\pi} \bigg[\frac{1}{\Gamma(N)} \int_0^{\infty} H_t^{\ab}(\t,\v) t^{N-1}\, dt \bigg]
    \, f(\v)\, d\mu_{\ab}(\v).
\end{align}
The application of Fubini's theorem in the last identity is justified
since $H_t^{\ab}(\t,\v) \ge 0$ and 
\begin{align}\label{iden0}
        \int_0^\pi {H}_t^{\ab}(\t,\v) \, d\mu_{\ab}(\v)
            =
\exp\big(-t(\mathcal{J}^{\ab})^{1/2} \big) \boldsymbol{1} (\t) 
=
 \exp\big(-t | \tau_{\ab} | \big), 
\qquad t>0, \quad \t \in (0,\pi);
\end{align}
here and elsewhere $\boldsymbol{1}$ is the constant function equal to $1$ on $(0,\pi)$. 
Thus we get
\begin{equation} \label{d7}
\big(\mathcal{J}^{\ab}\big)^{-N/2}f(\t) = \int_0^{\pi} K_{N/2}^{\ab}(\t,\v) f(\v)\, d\mu_{\ab}(\v),
    \qquad \t \in (0,\pi), \quad f \in C_c^{\infty}(0,\pi), \quad \tau_{\ab} \neq 0,
\end{equation}
where
$$
K^{\ab}_{\sigma}(\t,\v) = 
	\frac{1}{\Gamma(2\sigma)} \int_0^{\infty} H_t^{\ab}(\t,\v) t^{2\sigma-1}\, dt, 
	\qquad \t,\v \in (0,\pi), \quad	\sigma > 0, \quad
    \tau_{\ab} \neq 0.
$$
Note that sharp estimates for the potential kernel $K_{\sigma}^{\ab}(\t,\v)$, $\sigma>0$, were obtained
recently in \cite[Theorem 2.2]{NR}. Note also that the integral representation \eqref{d7} is valid for
$f\in L^2(d\mu_{\ab})$ (and even more general $f$), see \cite[Section 1]{NR}.

To deal with the case $\tau_{\ab} = 0$ we introduce the compensated Jacobi-Poisson kernel
\begin{equation}\label{Htilde}
\widetilde{H}_t^{\ab}(\t,\v) = H_t^{\ab}(\t,\v) - \exp\big( -t |\tau_{\ab}|\big)
    \P_0^{\ab}(\t) \P_0^{\ab}(\v), \qquad t>0, \quad \t,\v \in (0,\pi),
\end{equation}
which is essentially given by the series in \eqref{d0}, but with summation starting from $n=1$.
Observe that when $\tau_{\ab}=0$ the second term on the right-hand side here is simply a constant equal to
$1/\mu_{\ab}(0,\pi)$.
Also, $\widetilde{H}_t^{\ab}(\t,\v)$ has an exponential (and uniform in $\t$ and $\v$)
decay in $t \to \infty$, as easily seen by analyzing
the corresponding series. Repeating the previous arguments, for all $\ab>-1$ we get
$$
(\mathcal{J}^{\ab})^{-N/2} \Pi_0 f(\t) = \int_0^{\pi} \widetilde{K}_{N/2}^{\ab}(\t,\v)f(\v)\, d\mu_{\ab}(\v),
    \qquad \t \in (0,\pi), \quad f \in C_c^{\infty}(0,\pi),
$$
where
$$
\widetilde{K}_{\sigma}^{\ab}(\t,\v) =
    \frac{1}{\Gamma(2\sigma)} \int_0^{\infty} \widetilde{H}_t^{\ab}(\t,\v) t^{2\sigma-1}\, dt, 
    	\qquad \t,\v \in (0,\pi), \quad \sigma > 0,
$$
is the compensated potential kernel.

Our main results are the following (notice that the case $\tau_{\ab}=0$ 
is not distinguished in the statements).
\begin{thm} \label{thm:main}
Let $\ab > -1$ and $N \ge 1$. Then for each $f \in C_c^{\infty}(0,\pi)$ and all $\t \in (0,\pi)$
\begin{align*} 
R_{N}^{\ab}f(\t) & = \pv \int_0^{\pi} R_{N}^{\ab}(\t,\v) f(\v) \, d\mu_{\ab}(\v), \qquad 
\textrm{$N$ odd},\\ \nonumber
R_{N}^{\ab}f(\t) & = (-1)^{N/2} f(\t) + \int_0^{\pi} R_{N}^{\ab}(\t,\v) f(\v) \, d\mu_{\ab}(\v), \qquad
\textrm{$N$ even}.
\end{align*}
\end{thm}

Equipped with Theorem \ref{thm:main}, we establish an analogous representation for the extensions
of $R_N^{\ab}$ that are bounded on $L^p(wd\mu_{\ab})$, $w \in A_p^{\ab}$, $1< p < \infty$,
or from $L^1(wd\mu_{\ab})$ to weak $L^1(wd\mu_{\ab})$, $w \in A_1^{\ab}$. 

\begin{thm}\label{cor:mainmain}
Let $\ab > -1$ and $N \ge 1$. Assume that $1 \le p < \infty$, $w \in A_p^{\ab}$ and $f \in L^p(wd\mu_{\ab})$.
\begin{itemize}
\item[(a)]
If $N$ is odd, then
$$
R_N^{\ab}f(\t) = \pv \int_0^{\pi} R_N^{\ab}(\t,\v) f(\v)\, d\mu_{\ab}(\v), 
	\qquad \textrm{a.a.}\;\; \t \in (0,\pi).
$$
\item[(b)]
If $N$ is even, then
$$
R_N^{\ab}f(\t) = (-1)^{N/2} f(\t) + \int_0^{\pi} R_N^{\ab}(\t,\v)f(\v)\, d\mu_{\ab}(\v),
	\qquad \textrm{a.a.}\;\; \t \in (0,\pi).
$$
\end{itemize}
\end{thm}

This extends and refines the ultraspherical result \cite[Theorem 1.1]{BFRT}, where $\a = \b > -1/2$
and $\pv$ is always involved, independently of the order. For the first order ultraspherical
Riesz transform the principal value integral representation was obtained earlier in 
\cite[Theorem 2.13]{Bur1}, for polynomial functions and under a restriction on the ultraspherical
parameter of type. It is worth mentioning that, in the Jacobi context of this paper and with a restriction
on $\a$ and $\b$, \cite{Li} provides a definition of the conjugate function mapping and its
singular integral representation.
The latter is in a sense a kind of the first order Riesz-Jacobi transform, but differs from $R_1^{\ab}$.
Actually, the two operators arise by completely different motivations. The one in \cite{Li} 
goes back to the fundamental work \cite{MuS}, it is related
to the classical Fourier analysis on the torus and refers to connections between Fourier series, analytic
functions and harmonic functions. On the other hand, $R_1^{\ab}$ emerges from the `spectral' perspective
suggested in \cite{topics}, which offers a more natural background for defining higher order Riesz transforms.

As explained in \cite{NoSt}, see also \cite{L}, in some aspects there is a more
natural than $\delta^N$ notion of higher order derivative in the Jacobi setting given by interlacing
$\delta$ and $\delta^{*}$,
$$
{D}^N =
    \underbrace{\ldots \delta \delta^* \delta \delta^* \delta}_{N\; \textrm{components}}.
$$
Accordingly, as in \cite{L} we also consider Riesz-Jacobi transforms defined formally by
\begin{equation} \label{d8}
\mathcal{R}_N^{\ab} \colon f \mapsto  D^N \big(\J^{\ab}\big)^{-N/2}f.
\end{equation}
Similarly as in case of \eqref{d1},
this can easily be understood strictly on $\spann\{\P_n^{\ab}:n \ge 0\}$, with the convention concerning the
case $\tau_{\ab}=0$ in force. Consequently, we are led to operators defined on $L^2(d\mu_{\ab})$ by
$$
\mathcal{R}_N^{\ab}f = \sum_{n=1}^{\infty} \big|n+\tau_{\ab}\big|^{-N} 
	\big\langle f, \P_n^{\ab}\big\rangle_{d\mu_{\ab}}
	D^N\P_n^{\ab}.
$$
The last series indeed converges in $L^2(d\mu_{\ab})$, and $\mathcal{R}_N^{\ab}$, $N \ge 1$,
are bounded linear operators on $L^2(d\mu_{\ab})$, see \cite[Corollary 3.2]{L} and the relevant
arguments in the proof of \cite[Proposition 2.4]{L}.
As we shall see in Proposition \ref{prop:CZRi} below, $\mathcal{R}_N^{\ab}$
extends uniquely to a bounded operator on $L^p(wd\mu_{\ab})$, $w \in A_p^{\ab}$, $1< p < \infty$,
and from $L^1(wd\mu_{\ab})$ to weak $L^1(wd\mu_{\ab})$, $w \in A_1^{\ab}$. For these extensions, we prove
a representation analogous to that from Theorem~\ref{cor:mainmain}. Denote
$$
\mathcal{R}_N^{\ab}(\t,\v) = \frac{1}{\Gamma(N)} \int_0^{\infty} D_{\t}^{N} H_t^{\ab}(\t,\v) t^{N-1}\, dt
$$
(the last integral converges absolutely for $\t \neq \v$, see e.g.\ Lemma~\ref{lem:thetaJP}
in Section~\ref{sec:proof}).
\begin{thm}\label{thm:main2}
Let $\ab > -1$ and $N \ge 1$. Assume that $1 \le p < \infty$, $w \in A_p^{\ab}$ and $f \in L^p(wd\mu_{\ab})$.
\begin{itemize}
\item[(a)]
If $N$ is odd, then
$$
\mathcal{R}_N^{\ab}f(\t) = \pv \int_0^{\pi} \mathcal{R}_N^{\ab}(\t,\v) f(\v)\, d\mu_{\ab}(\v), 
	\qquad \textrm{a.a.}\;\; \t \in (0,\pi).
$$
\item[(b)]
If $N$ is even, then
$$
\mathcal{R}_N^{\ab}f(\t) =  f(\t) + \int_0^{\pi} \mathcal{R}_N^{\ab}(\t,\v)f(\v)\, d\mu_{\ab}(\v),
	\qquad \textrm{a.a.}\;\; \t \in (0,\pi).
$$
\end{itemize}
\end{thm}

It is remarkable that $\mathcal{R}_N^{\ab}$ possesses in general better mapping properties than
$R_N^{\ab}$, which apparently has not been noticed earlier. This phenomenon is probably best seen
from the $L^1$ behavior. In the next section we will show that $\mathcal{R}_{2k}^{\ab}$, $k \ge 1$,
are bounded on $L^1(d\mu_{\ab})$, see Remark~\ref{rem:bndRe}. 
On the other hand, in Appendix we prove the following.
\begin{pro}\label{pro:unbound}
Let $\ab > -1$, $(\a,\b)\neq (-1/2,-1/2)$. Then $R_2^{\ab}$ is not bounded on $L^1(d\mu_{\ab})$.
\end{pro}
In fact, a similar negative result holds also for any $R_{2k}^{\ab}$, $k\ge 2$,
but the proof is much more involved and hence beyond the scope of this paper.

\section{Proofs of the main results} \label{sec:proof}

We begin with some preparatory results.
\begin{lem} \label{lem:diffK}
Let $\ab > -1$, $N \ge 1$ and $j \ge 0$. For $\t,\v \in (0,\pi)$, $\t\neq \v$, we have
$$
\partial_{\t}^{j} K^{\ab}_{N/2}(\t,\v) =
\frac{1}{\Gamma(N)} \int_0^{\infty} \partial_{\t}^{j} H_t^{\ab}(\t,\v) t^{N-1}\, dt,
$$
where one should replace $K^{\ab}_{N/2}(\t,\v)$ by $\widetilde{K}_{N/2}^{\ab}(\t,\v)$
and $H_t^{\ab}(\t,\v)$ by $\widetilde{H}_t^{\ab}(\t,\v)$
in case $\tau_{\ab} = 0$
(actually, after this replacement the formula is valid for all $\ab > -1$).
\end{lem}

\begin{proof}
We focus on the case $\tau_{\ab} \neq 0$; the opposite one is analogous.
Assume that $\v$ is fixed and $K$ is a closed interval contained in $(0,\pi)$ such that $\v \notin K$. Using \cite[Lemma 3.8]{NSS} and \cite[Corollary 3.5]{NSS} together with the bound $\q \gtrsim (\t-\vp)^2$
for $\q$ appearing there, see \eqref{q_for} below, we arrive at the estimate
\begin{align*}
| \partial_{\t}^j H_t^{\ab}(\t,\v) |
    \lesssim e^{-ct}, \qquad t>0, \quad \t \in K,
\end{align*}
with $c=c_{\ab}>0$.
This allows us to apply the dominated convergence theorem and the conclusion follows.
\end{proof}

Taking $j=N \ge 1$ in Lemma \ref{lem:diffK} we see that the order of integration and differentiation
in the definition of the Riesz-Jacobi kernel in \eqref{d4} can be exchanged.
\begin{cor}\label{cor:Rker}
Let $\ab > -1$ and $N \ge 1$. For $\t,\v \in (0,\pi)$, $\t\neq \v$, we have
$$
R_N^{\ab}(\t,\v) = \partial_{\t}^{N} \Bigg[\frac{1}{\Gamma(N)} \int_0^{\infty} {H}_t^{\ab}(\t,\v)
    t^{N-1}\, dt\Bigg],
$$
where one should replace $H_t^{\ab}(\t,\v)$ by $\widetilde{H}_t^{\ab}(\t,\v)$ in case $\tau_{\ab} = 0$
(actually, after this replacement the formula is valid for all $\ab > -1$).
\end{cor}

The next lemma can be regarded as an extension of \cite[Step 1 on p.\,516]{BFRT}.
It will be crucial in reducing the proof of Theorem~\ref{thm:main} to the case $N=1$.
\begin{lem}\label{lem:Step1}
Let $\ab > -1$, $N \ge 1$ and $0 \le j \le N-1$. Then for $f \in C_c^\infty(0,\pi)$ and $\t \in (0,\pi)$
we have
\[
\partial_{\t}^j \big(\J^{\ab}\big)^{-N/2} f (\t)
= \int_0^{\pi} \partial_{\t}^j K^{\ab}_{N/2} (\t,\v) f(\v)
\, d\mu_{\ab}(\v), 
\]
where one should replace $K^{\ab}_{N/2} (\t,\v)$ by
$\widetilde{K}^{\ab}_{N/2} (\t,\v)$ and $f$ on the left-hand side by $\Pi_0 f$ in case $\tau_{\ab} = 0$
(actually, after this replacement the formula is valid for all $\ab > -1$).
\end{lem}
The proof of Lemma \ref{lem:Step1} will be given in Section~\ref{sec:thetaJP}.
The reasoning involves Lemma \ref{lem:diffK}, as well as suitable estimates of 
$\partial_{\t}^j H_t^{\ab} (\t,\v)$, $j \ge 0$. The latter will also be used directly in the proof of 
Theorem~\ref{thm:main} and are provided by the next lemma. This result is of independent interest,
and its proof is located also in Section~\ref{sec:thetaJP}.
\begin{lem}\label{lem:thetaJP}
Let $\ab > -1$ and $j \ge 0$. Then
\begin{align*}
|\partial_{\t}^j H_t^{\ab} (\t,\v)|
\lesssim
\Big( t^2+ \t^2+\v^2 \Big)^{-\a-1/2}
    \Big( t^2 + (\pi-\t)^2 + (\pi-\v)^2 \Big)^{-\b-1/2}
\frac{t}{\big[ t^2+(\t-\v)^2 \big]^{1+j/2}},
\end{align*}
uniformly in $ t \le 1$ and $\t,\v \in (0,\pi)$.
Further, excluding the case when $j = \tau_{\ab} = 0$, there exists $c = c_{\ab}>0$ such that 
\[
|\partial_{\t}^j H_t^{\ab} (\t,\v)|
\lesssim
e^{-ct},
\qquad t \ge 1, \quad \t,\v \in (0,\pi).
\]
This estimate holds with no restrictions on $\ab$ and $j$ if $H_t^{\ab}(\t,\v)$ is replaced by
$\widetilde{H}_t^{\ab}(\t,\v)$.
\end{lem}

The result below says that Theorem \ref{thm:main} holds in the special case when $N=1$ and 
$f = \boldsymbol{1}$. 
Its proof is rather long and technical, hence is postponed to Section~\ref{sec:PV0}.
\begin{lem}\label{lem:PV=0}
    Let $\ab >-1$. Then
    $$
\pv \int_{0}^\pi R_1^{\ab}(\t,\v) \, d\mu_{\ab}(\v) = 0, \qquad \t \in (0,\pi).
$$
\end{lem}

We are now in a position to prove Theorem~\ref{thm:main}.

\begin{proof}[Proof of Theorem~\ref{thm:main}.]
Our reasoning has two steps: (1.) reduction to the case $N=1$ and (2.) the proof for $N=1$.
Let $f \in C_c^{\infty}(0,\pi)$ be fixed.

\noindent \textbf{Step 1.} 
For the sake of clarity let us assume that $\tau_{\ab} \neq 0$; the parallel case $\tau_{\ab}=0$ is
commented at the end of Step 1.
Observe that, in view of the explicit formula for $\J^{\ab}$, we have the decompositions
\begin{align} \label{dec1}
\partial_{\t}^{2m} & = (-1)^m  \big(\J^{\ab}\big)^m + \sum_{0 \le j < 2m} f_{m,j}(\t) \partial_{\t}^j,
	\qquad m \ge 1,\\ \nonumber
\partial_{\t}^{2m+1} & = (-1)^m \partial_{\t} \big(\J^{\ab}\big)^m + \sum_{0 \le j \le 2m} g_{m,j}(\t) \partial_{\t}^j, \qquad m \ge 0, 
\end{align}
where $f_{m,j}$ and $g_{m,j}$ are some smooth, possibly unbounded, functions of $\t \in (0,\pi)$
(of course, $f_{m,j}$ and $g_{m,j}$ can be determined explicitly, but we shall not need this).
Thus, in view of \eqref{d6}, for all $\t \in (0,\pi)$
\begin{align}
R_{2m}^{\ab}f(\t) & = (-1)^m f(\t) + \sum_{0 \le j < 2m} f_{m,j}(\t) \partial_{\t}^j \big(\J^{\ab}\big)^{-m}f(\t),
\qquad m \ge 1, \quad \tau_{\ab} \neq 0,\label{decReven}\\
R_{2m+1}^{\ab}f(\t) & = (-1)^m R_1^{\ab} f(\t) + \sum_{0 \le j \le 2m} g_{m,j}(\t)\partial_{\t}^j
	\big(\J^{\ab}\big)^{-m-1/2}f(\t), \qquad m \ge 0,\quad \tau_{\ab} \neq 0,\label{decRodd}
\end{align}
where we used the identities
\begin{align} \label{idenJ1}
(\J^{\ab})^{m}(\J^{\ab})^{-m}f(\t) & = f(\t), \qquad m \ge 1, \quad \tau_{\ab} \neq 0,\\ 
(\J^{\ab})^{m}(\J^{\ab})^{-m-1/2}f(\t) & = (\J^{\ab})^{-1/2}f(\t), \qquad m \ge 0, \quad \tau_{\ab} \neq 0.  \label{idenJ2}
\end{align}
The latter are valid for all $\t \in (0,\pi)$ and can be
verified by applying $(\J^{\ab})^m$ under the series defining $(\J^{\ab})^{-m}$ or
$(\J^{\ab})^{-m-1/2}$, respectively, see \eqref{d2}, Proposition~\ref{prop:FJdecay} and \eqref{d5}.

Taking into account Lemma \ref{lem:Step1} we infer that the sum over $0 \le j < 2m$ appearing in \eqref{decReven}
has a usual integral representation and the corresponding kernel is
\begin{equation} \label{comp7}
\sum_{0 \le j < 2m} f_{m,j}(\t) \partial_{\t}^{j} K_m^{\ab}(\t,\v) = \partial_{\t}^{2m} K_m^{\ab}(\t,\v)
	-(-1)^m \big(\J^{\ab}\big)_{\t}^{m} K_m^{\ab}(\t,\v).
\end{equation}
Therefore, in view of Corollary \ref{cor:Rker}, 
to finish proving Theorem \ref{thm:main} for $N$ even it suffices to check that 
\begin{equation} \label{HarmK}
\big(\J^{\ab}\big)_{\t}^m K_m^{\ab} (\t,\v) = 0,
\qquad  \t,\v \in (0,\pi), \quad \t \ne \v, \quad m \ge 1, \quad \tau_{\ab} \neq 0.
\end{equation}
Since $H_t^{\ab}(\t,\v)$ satisfies the Poisson equation based on $\J^{\ab}$, we have
$$
\big(\J^{\ab}\big)^m_{\t} H_t^{\ab} (\t,\v) = \partial_t^{2m} H_t^{\ab} (\t,\v);
$$
this identity can also be seen directly, by differentiating the series in \eqref{d0}.
Thus an application of Lemma~\ref{lem:diffK} and then integration by parts give us
\begin{equation} \label{HermKe}
\big(\J^{\ab}\big)_{\t}^m K_m^{\ab} (\t,\v)
    =  \frac{1}{\Gamma(2m)} \int_0^{\infty} \partial_t^{2m} H_t^{\ab}(\t,\v) t^{2m-1}\, dt = 0, 
    \qquad \t \ne \v, \quad m \ge 1, \quad \tau_{\ab} \neq 0;
\end{equation}
here to ensure that
$\partial_t^j H_t^{\ab}(\t,\v) t^j$, $0 \le j < 2m$, vanish as $t\to 0^+$ and as
$t \to \infty$,
we use 
\cite[Corollary~3.5, Theorem 6.1 and Lemma 3.8]{NSS}. 

Similarly, using again Lemma \ref{lem:Step1} we see that the sum over $0 \le j \le 2m$ appearing in \eqref{decRodd} has a usual integral representation, the kernel being given by
\begin{equation} \label{ff}
\sum_{0 \le j \le 2m} g_{m,j}(\t) \partial_{\t}^j K_{m+1/2}^{\ab}(\t,\v) =
\partial_{\t}^{2m+1} K_{m+1/2}^{\ab}(\t,\v) - (-1)^m \partial_{\t} 
\big(\J^{\ab}\big)_{\t}^m K^{\ab}_{m+1/2}(\t,\v).
\end{equation}
But here $(\J^{\ab})_{\t}^m K^{\ab}_{m+1/2}(\t,\v) = K_{1/2}^{\ab}(\t,\v)$, $\t \neq \v$, 
which can be checked
similarly as \eqref{HarmK}. 
This, by means of Corollary~\ref{cor:Rker}, proves the theorem for $N$ odd, assuming that it holds for $N=1$.

When $\tau_{\ab}=0$, one has to replace $f$ by $\Pi_0 f$ on the right-hand side of
\eqref{decReven} and also in the sum over $0 \le j \le 2m$ on the right-hand side of \eqref{decRodd}. The same modification is needed on both sides of \eqref{idenJ1} and \eqref{idenJ2}, and $K_m^{\ab}(\t,\v)$ should be replaced by $\widetilde{K}_m^{\ab}(\t,\v)$ in \eqref{comp7}.
Further, instead of \eqref{HarmK} we have 
$$
\big(\J^{\ab}\big)_{\t}^m \widetilde{K}_m^{\ab}(\t,\v) = -1/\mu_{\ab}(0,\pi), 
\qquad \t,\v \in (0,\pi), \quad \t \neq \v, \quad m \ge 1, \quad \tau_{\ab} = 0,
$$
since the last integration by parts in an analogue of \eqref{HermKe} with $K_{m}^{\ab}(\t,\v)$ replaced
by $\widetilde{K}_m^{\ab}(\t,\v)$ gives
$-H_t^{\ab}(\t,\v)\big|_{t=\infty} = - 1/\mu_{\ab}(0,\pi)$, $\t \ne \v$. Taking into account that
$$
\Pi_0f(\t) = f(\t) - \frac{1}{\mu_{\ab}(0,\pi)} \int_0^{\pi} f(\v)\, d\mu_{\ab}(\v), \qquad \t \in (0,\pi),
$$
we conclude the theorem for $N$ even also in case $\tau_{\ab}=0$.
Finally, replacing in \eqref{ff} $K^{\ab}_{m+1/2}(\t,\v)$ and $K_{1/2}^{\ab}(\t,\v)$ by $\widetilde{K}^{\ab}_{m+1/2}(\t,\v)$ and $\widetilde{K}_{1/2}^{\ab}(\t,\v)$, respectively, makes
the reasoning of Step 1 for $N$ odd go through in case $\tau_{\ab}=0$.

\noindent \textbf{Step 2.} 
Let $N=1$ and $\t \in (0,\pi)$ be fixed.
Proceeding as in the chain of identities \eqref{ch1}, i.e.\ using Proposition~\ref{prop:FJdecay}, \eqref{d5} and Fubini's theorem, we get
\begin{equation}\label{intRiesz}
R_1^{\ab} f (\theta)
        = \int_0^\infty \int_0^{\pi} \partial_{\t} {H}_t^{\ab}(\t,\v) f (\v) \, d\mu_{\ab}(\v) \, dt;
\end{equation} 
here the case $\tau_{\ab} = 0$ is also included. 
The above double integral is not absolutely convergent, so one cannot use Fubini's 
theorem to change the order of integration. However, taking into account \eqref{iden0},
for each $t>0$, the dominated convergence theorem gives us
    \begin{equation} \label{zeroint}
        \int_0^{\pi} \partial_{\t} {H}_t^{\ab}(\t,\v) \, d\mu_{\ab}(\v) = 0, \qquad t>0.
    \end{equation}
This identity combined with \eqref{intRiesz} leads to
\[
R_1^{\ab} f (\theta)
        = \int_0^\infty \int_0^{\pi} \partial_{\t} {H}_t^{\ab}(\t,\v) 
\big[f (\v) - f(\t)\big] \, d\mu_{\ab}(\v) \, dt.
\]
We claim that this double integral is absolutely convergent. Indeed, using Lemma~\ref{lem:thetaJP}
and the Mean Value Theorem we get
    \begin{align*}
\int_0^{\pi} \bigg( \int_0^1 + \int_1^\infty \bigg) \big|\partial_{\t} {H}_t^{\ab}(\t,\v)\big| |f (\v) - f(\t)| \, dt \, d\mu_{\ab}(\v) 
\lesssim 1 + \int_0^{\pi} \int_0^1 
\frac{|\t-\v| t \, dt}{(t+|\theta-\v|)^3}  \, d\mu_{\ab}(\v)
            < \infty,
    \end{align*}
where the last inequality follows by the change of variable $s=|\t-\v| t$.
    Hence an application of Fubini's theorem produces
    \begin{align*}
        R_1^{\ab} f (\theta)
             = \int_0^{\pi}  R_1^{\ab}(\t,\v) \big[f (\v)-f(\t)\big] \, d\mu_{\ab}(\v).
    \end{align*}
This together with Lemma \ref{lem:PV=0} concludes Step 2. 

The proof of Theorem \ref{thm:main} is finished.
\end{proof}

Next, we give the proof of Theorem \ref{cor:mainmain}. In order to show item (b), we will need
the following technical result whose proof is located in Section~\ref{sec:weightRiesz}.

\begin{lem}\label{lem:weightRiesz}
Assume that $\ab > -1$.
Let $N \ge 2$ be even and let $\t \in (0,\pi)$ be fixed.
\begin{itemize}
\item[(i)]
If $1 \le p < \infty$, then $(0,\pi) \ni \v \mapsto R_N^{\ab}(\t,\v)$ belongs to $L^p(wd\mu_{\ab})$, $w \in A_p^{\ab}$.
\item[(ii)]
If $p=1$ and $w \in A_1^{\ab}$, then $(0,\pi) \ni \v \mapsto R_N^{\ab}(\t,\v)/w(\v)$ belongs to $L^\infty(0,\pi)$.
\end{itemize}
\end{lem}

\begin{proof}[Proof of Theorem \ref{cor:mainmain}.]
Since $R_N^{\ab}$ is a Calder\'on-Zygmund operator, see \cite[Theorem 5.1]{NSS}, 
item (a) is a consequence of Theorem~\ref{thm:main} and standard arguments based on the 
Calder\'on-Zygmund theory. A crucial point are weighted $L^p$-boundedness properties of the truncated
integrals maximal operator associated with $R_N^{\ab}$. See
\cite[p.\,515]{BFRT} for the details given in the ultraspherical situation. 

To prove item (b) we fix $f \in L^p(wd\mu_{\ab})$, $w \in A_p^{\ab}$, $1 \le p < \infty$, 
and take a sequence $\{ f_n \} \subset C_c^{\infty}(0,\pi)$ approximating $f$ in $L^p(wd\mu_{\ab})$.
We may assume that $f_n$ converges to $f$ also pointwise a.e. 
Noting that $w \in A_p^{\ab}$ implies $w^{-p'/p} \in A_{p'}^{\ab}$, $1<p<\infty$, $1/p+1/p'=1$, and using
Lemma~\ref{lem:weightRiesz} (item (i) for $1<p<\infty$, and item (ii) in case $p=1$) we see that
\[
\int_0^{\pi} R_N^{\ab}(\t,\v) f_{n}(\v)\, d\mu_{\ab}(\v)
\xrightarrow{n \to \infty}
\int_0^{\pi} R_N^{\ab}(\t,\v) f(\v)\, d\mu_{\ab}(\v), 
\qquad \t \in (0,\pi).
\]
This, together with Theorem~\ref{thm:main} and the boundedness properties of $R_N^{\ab}$, 
implies the desired conclusion.
\end{proof}
Finally, we justify Theorem \ref{thm:main2}. For this purpose, we need the following result which for the
restricted range $\ab \ge -1/2$ was obtained in \cite{L}.

\begin{pro} \label{prop:CZRi}
Let $\ab > -1$ and let $N \ge 1$. Then $\mathcal{R}_N^{\ab}$ is a Calder\'on-Zygmund operator in the sense
of the space of homogeneous type $((0,\pi),d\mu_{\ab},|\cdot|)$, associated with the kernel 
$\mathcal{R}_N^{\ab}(\t,\v)$. In particular, $\mathcal{R}_N^{\ab}$ 
extends uniquely from $L^2(d\m_{\ab})\cap L^p(wd\mu_{\ab})$ to a bounded operator on
$L^p(wd\mu_{\ab})$, $w\in A_p^{\ab}$, $1<p<\infty$, and, when $p=1$,
to a bounded operator from $L^1(wd\mu_{\ab})$ to weak $L^1(wd\mu_{\ab})$, $w \in A_1^{\ab}$.
\end{pro}

Since $\mathcal{R}_N^{\ab}$ coincides, up to a multiplicative constant, with $\big(R_N^{\ab}\big)^+$ investigated in \cite{L}, for $\ab \ge -1/2$ Proposition~\ref{prop:CZRi} 
is covered by \cite[Theorem 2.3]{L}. To deal with all $\ab>-1$ we present a unified treatment based on the technique developed recently in \cite{NSS} and on the already known case of the first order Riesz-Jacobi transform $\mathcal{R}_1^{\ab} = R_1^{\ab}$, see \cite[Theorem 5.1]{NSS}. From the present perspective this method seems to be more natural than the one in \cite{L}. Nevertheless, 
proceeding in the spirit of \cite{L} allows to achieve Proposition~\ref{prop:CZRi} as well.

\begin{proof}[Proof of Proposition~\ref{prop:CZRi}.]
Observe that $D^2 = \delta^*\delta= \J^{\ab}-\tau_{\ab}^2$ and so we have the decompositions
\begin{align} \label{d9}
    D^{2m} & =  
    \sum_{j=0}^{m} \binom{m}{j} (-\tau_{\ab}^2)^{j} \big(\J^{\ab}\big)^{m-j}, \qquad m \ge 1, \\ \nonumber
	 D^{2m + 1} & = 
	 \sum_{j=0}^{m} \binom{m}{j} (-\tau_{\ab}^2)^{j} \partial_{\t} \big(\J^{\ab}\big)^{m-j}, \qquad m \ge 0.
\end{align}
Thus, for $f \in C_c^{\infty} (0,\pi)$ and $\t \in (0,\pi)$, we get
\begin{align}
\mathcal{R}_{2m}^{\ab} f(\t) & =  
f(\t) + \sum_{j=1}^m \binom{m}{j}(-\tau_{\ab}^2)^j 
\big( \J^{\ab}\big)^{-j} f(\t), 
\qquad m \ge 1, \quad \tau_{\ab} \ne 0, \label{decRec}\\
\mathcal{R}_{2m+1}^{\ab} f(\t) & = 
R_1^{\ab} f(\t) +
    \sum_{j=1}^m \binom{m}{j} (-\tau_{\ab}^2)^j \partial_{\t} \big(\J^{\ab}\big)^{-j-1/2} f(\t),
\qquad m \ge 0, \quad \tau_{\ab} \ne 0; \label{decRoc}
\end{align}
here we used the identities \eqref{idenJ1} and \eqref{idenJ2}. 
The case $\tau_{\ab} = 0$ is even easier and we simply get
$$
\mathcal{R}_{2m}^{\ab} f = \Pi_0 f, \quad m \ge 1, \qquad \textrm{and} \qquad 
\mathcal{R}_{2m+1}^{\ab} f = R_1^{\ab} f, \quad m \ge 0. 
$$

By \cite[Theorem 5.1]{NSS} we know that $R_1^{\ab}$ is a Calder\'on-Zygmund operator, hence our task reduces to showing that 
$$
\big( \J^{\ab}\big)^{-j} \quad \textrm{and} \quad \partial_{\t} \big( \J^{\ab}\big)^{-j - 1/2}, 
\qquad j \ge 1, \quad \tau_{\ab} \ne 0,
$$
are Calder\'on-Zygmund operators associated with the space $((0,\pi),d\mu_{\ab},|\cdot|)$,
with the corresponding kernels $K_j^{\ab}(\t,\v)$ and $\partial_{\t}K_{j+1/2}^{\ab}(\t,\v)$, respectively.
In the first case the conclusion  
is a consequence of a more general result. Namely, the potential operators 
$\big( \J^{\ab}\big)^{-\sigma}$, $\sigma > 0$, $\tau_{\ab} \ne 0$, are special instances of multipliers
of Laplace-Stieltjes transform type investigated in \cite{NSS}, the related measure being 
$d\nu(t) = \frac{1}{\Gamma(\sigma)} t^{\sigma - 1} \, dt$. Then the assumption \cite[(18)]{NSS}
is trivially satisfied and so 
\cite[Theorem 5.1]{NSS} implies the desired property.

In the second case we first observe that the operators in question are bounded on $L^2(d\mu_{\ab})$. This
can be seen by means of \eqref{d2}, Proposition~\ref{prop:FJdecay}, \eqref{derP} and the fact that
$\{\frac{1}2\sin\t \P_n^{\a+1,\b+1}:n \ge 0\}$ is an orthonormal basis in $L^2(d\mu_{\ab})$.
Further, applying Lemmas~\ref{lem:Step1} and \ref{lem:diffK} we see that they 
are integral operators in the usual sense with the kernels
\begin{align} \label{ker}
\partial_{\t} K_{j+ 1/2}^{\ab} (\t,\v)
=
\frac{1}{\Gamma(2j + 1)} \int_0^\infty \partial_{\t} H_t^{\ab} (\t,\v) t^{2j}
\, dt, \qquad \t,\v \in (0,\pi), \quad \t \ne \v, \quad j \ge 1.
\end{align}
This gives, in particular, kernel associations in the Calder\'on-Zygmund theory sense. 
It remains to show the standard estimates, i.e.\ \cite[(15), (19)]{NSS} with $\mathbb{B} = \mathbb{C}$,
for these kernels.
This, however, can be done in a straightforward manner by employing the method established in \cite{NSS}, 
see the proof of \cite[Theorem~4.1]{NSS}. 
For reader's convenience we now indicate the main steps.

We split the region of integration in \eqref{ker} onto $(0,1)$ and $(1,\infty)$, and treat 
the resulting parts separately. The latter part can be analyzed by means of \cite[Corollary 3.9]{NSS}. 
On the other hand, the remaining part can be handled by using \cite[Corollary~3.5 and Lemma 3.7]{NSS} 
and the boundedness of $\q$ there, see \eqref{q_for}. 

This finishes showing that $\mathcal{R}_{N}^{\ab}$, $N \ge 1$, are Calder\'on-Zygmund operators. 
The fact that $\mathcal{R}_{N}^{\ab} (\t,\v)$ is
the Calder\'on-Zygmund kernel of $\mathcal{R}_{N}^{\ab}$ 
becomes clear after tracing in detail
the proof of Theorem~\ref{thm:main2} sketched below.
\end{proof}

\begin{proof}[Proof of Theorem \ref{thm:main2}.]
We first verify the result for $f \in C_c^{\infty}(0,\pi)$ and thus proceed as in the proof of 
Theorem~\ref{thm:main}. Then it is enough to reduce the problem to the case $N=1$ since the integral
representation for $\mathcal{R}_1^{\ab} = R_1^{\ab}$ is already provided by Theorem~\ref{thm:main}.
To obtain the desired reduction one repeats the arguments from Step~1 of the proof of Theorem~\ref{thm:main},
with the aid of Lemmas~\ref{lem:diffK} and \ref{lem:Step1}, and the decompositions 
\eqref{decRec} and \eqref{decRoc}.

Having the representations of the theorem for all $f \in C_c^{\infty}(0,\pi)$, we pass to a general
$f$ as in the proof of Theorem~\ref{cor:mainmain}, by means of Proposition~\ref{prop:CZRi} and
an analogue of Lemma~\ref{lem:weightRiesz}. The latter easily follows by \eqref{d9} and the arguments
from the proof of Lemma~\ref{lem:weightRiesz}.
\end{proof}

\begin{rem} \label{rem:bndRe}
The decomposition \eqref{decRec} combined with \cite[Theorem 2.3]{NR} reveals that $\mathcal{R}_N^{\ab}$,
for $N$ even, are bounded on $L^1(d\mu_{\ab})$ and on $L^{\infty}(0,\pi)$.
\end{rem}

\section{Proofs of Lemmas \ref{lem:thetaJP} and \ref{lem:Step1}} \label{sec:thetaJP}

To prove Lemma~\ref{lem:thetaJP} we will need some preparatory results that were obtained in the previous papers \cite{NoSj2,NSS}.
To state them we shall use the same notation as in \cite{NSS}.
For $\a >-1/2$ we denote by $d\Pi_\a$ the probability measure on the interval $[-1,1]$ given by the density
$$
\pia = \frac{\Gamma(\a+1)}{\sqrt{\pi}\Gamma(\a+1/2)} \, \big(1-u^2\big)^{\a-1/2}\, du,
$$
and in the limit case $d\Pi_{-1/2}$ is the sum of point masses at $-1$ and $1$ divided by $2$.
Further, let
$$
d\abpiK
= d\Pi_{(\a+1)K - (1-K)/2}
= \begin{cases}
        d\Pi_{-1\slash 2}, & K=0 \\
        d\Pi_{\a + 1}, &  K=1
    \end{cases}, \qquad \a >-1,
$$
and put
\begin{equation} \label{q_for}
\q = q(\t,\v,u,v) = 1 - u \st \svp - v \ct \cvp,
\qquad \t,\v \in (0,\pi), \quad u,v \in [-1,1].
\end{equation}
For further reference, we also introduce the function
\[
\Pi_{\a} (u) = 
\frac{\Gamma(\a+1)}{\sqrt{\pi}\Gamma(\a+1/2)} 
\int_0^u \big(1-w^2\big)^{\a-1/2}\, dw,
\qquad -1 < u < 1, \quad -1 < \a < -1/2,
\]
which is odd in $(-1,1)$ and negative for $u \in (0,1)$.
Moreover, for each fixed $-1 < \a < -1/2$, see \cite[Lemma~2.2]{NSS}, it satisfies
\begin{equation}\label{smallab}
 |\Pi_{\a}  (u) | \, du \simeq |u| d\Pi_{\a + 1} (u), \qquad u \in (-1,1).
\end{equation}
In the sequel we will frequently use, sometimes without mentioning, the following elementary relations, see \cite[p.\,738]{NoSj},
\begin{equation}
\begin{split} \label{elem}
1 - \ct \cvp & = \sin^2\frac{\t - \v}{4} + \sin^2\frac{\t + \v}{4} \simeq \t^2 + \v^2, 
\qquad \t,\v \in (0,\pi),\\
1 - \st \svp & = \sin^2\frac{\t - \v}{4} + \cos^2\frac{\t + \v}{4} \simeq (\pi - \t)^2 + (\pi - \v)^2, 
\qquad \t,\v \in (0,\pi),\\
1 - \st \svp - \ct \cvp & = 2 \sin^2\frac{\t - \v}{4} \simeq (\t - \v)^2, 
\qquad \t,\v \in (0,\pi).
\end{split}
\end{equation}

\begin{lem}[{\cite[Lemma A.2]{NoSj2}}, {\cite[Lemma 3.2 (a)]{NSS}}]
\label{lem:intest}
Let $\kappa \ge 0$ and $\gamma$, $\nu$ be such that $\gamma > \nu +1/2 \ge 0$. Then
$$
\int \frac{d\Pi_{\nu}(s)}{(D-Bs)^{\kappa}(A-Bs)^{\gamma}} \simeq
    \frac{1}{(D-B)^{\kappa} A^{\nu+1/2} (A-B)^{\gamma-\nu-1/2}}, \qquad 0 \le B < A \le D.
$$
\end{lem}

The next result is a refined specification of \cite[Corollary 3.5]{NSS}. 
The absence of
differentiations with respect to $t$ and $\v$ allows for 
more precise estimates than those established in \cite{NSS}.

\begin{lem}\label{lem:tNSS}
Let $j \ge 1$ be fixed.
Then the following estimates hold uniformly in $t \le 1$ and
$\t,\v \in (0,\pi)$.
\begin{itemize}
\item[(i)]
If $\a,\b \ge -1\slash 2$, then
\begin{align*}
\big| \partial_\t^j
H_{t}^{\ab}(\t,\v)
\big|
    \lesssim
t
\iint
\frac{\pia \, \pib}{(t^2 + \q)^{   \a + \b + 2 + j\slash 2  }}.
\end{align*}
\item[(ii)]
If $-1 < \a < -1\slash 2 \le \b$, then
\begin{align*}
\big|
     \partial_\t^j
H_{t}^{\ab}(\t,\vp)
\big|
    \lesssim
t
\sum_{K=0,1}
\sum_{k=1,2}
\bigg( \st+\svp \bigg)^{Kk}
\iint
\frac{\piK \, \pib}{(t^2 + \q)^{   \a + \b +  2 + (j + Kk)\slash 2  }}.
\end{align*}
\item[(iii)]
If $-1 < \b < -1\slash 2 \le \a$, then
\begin{align*}
\big|
    \partial_\t^j
H_{t}^{\ab}(\t,\v)
\big|
    \lesssim
t
\sum_{R=0,1}
\sum_{r=1,2}
\bigg( \ct + \cvp \bigg)^{Rr}
\iint
\frac{\pia \, \piR}{(t^2 + \q)^{   \a + \b + 2 + (j + Rr)\slash 2  }}.
\end{align*}
\item[(iv)]
If $-1 < \a,\b < -1\slash 2$, then
\begin{align*}
\big|
    \partial_\t^j
H_{t}^{\ab}(\t,\v)
\big|
    &\lesssim
t
\sum_{K,R=0,1}
\sum_{k,r=1,2}
\bigg( \st+\svp \bigg)^{Kk} \bigg( \ct + \cvp \bigg)^{Rr}
\\
& \qquad \qquad \times
\iint
\frac{\piK \, \piR}{(t^2 + \q)^{   \a + \b + 2 + (j + Kk +Rr)\slash 2  }}.
\end{align*}
\end{itemize}
\end{lem}

\begin{proof}
The reasoning is almost a repetition of the proof of \cite[Corollary 3.5]{NSS}. The only difference is that one should use an improvement of a special case of \cite[(11)]{NSS} instead of
\cite[Lemma 3.3]{NSS}. For reader's convenience we give some details, however, for any unexplained symbols we refer to \cite{NSS}.

First, observe that for each $j \ge 1$ the quantity $\partial_\t^j H_{t}^{\ab}(\t,\v)$ coincides with $\partial_\t^j \mathbb{H}_{t}^{\ab}(\t,\v)$. Further, $\Puv$ appearing in \cite[Proposition 2.3]{NSS} is equal to a constant times
$\sinh(\frac{t}2) \widetilde{\Psi}^{\a + \b +2}(t,\mathfrak{q})$.
Furthermore, assume for a moment that for $L=0$ the estimate \cite[(11)]{NSS} can be improved by restricting the summation on the right-hand side there to $k,r=1,2$.
Then combining
these facts with \cite[Proposition 2.3 and Lemma 2.2]{NSS} 
and the estimate $\sinh\frac{t}2 \simeq t$, $t \le 1$, we get the required bounds.
Hence, in order to finish the proof, it suffices to justify the above mentioned improvement of \cite[(11)]{NSS} for $L=0$. The details are as follows.

We proceed as in the proof of \cite[(11)]{NSS} and observe that the condition $L=0$ forces $k_2=r_2=k_5=r_5=0$ there. This leads to the last estimate in the proof of \cite[Lemma~3.3]{NSS} but with the summation restricted to $r_1+r_3+r_4=R$, $k_1+k_3+k_4=K$.
These constraints imply $2r_1 + 2r_3 + r_4 \in \{R,2R\}$ and $2k_1 + 2k_3 + k_4 \in \{K,2K\}$.
Since $\q$ is bounded, the conclusion follows.
\end{proof}

\begin{proof}[Proof of Lemma \ref{lem:thetaJP}.]
Since the estimate for large $t$ is a special case of \cite[Lemma 3.8]{NSS}, we focus on proving the bound for $t \le 1$.
The reasoning is based on the technique used in the proofs of \cite[Theorem~A.1]{NoSj2} and \cite[Theorem 6.1]{NSS}. Since the case $j=0$ is contained in the latter result, we assume that $j \ge 1$.
Further, we will show the desired estimate in the most involved case $-1 < \ab < -1/2$; the proofs of the remaining cases are similar and hence are omitted.

Notice that Lemma~\ref{lem:tNSS} reduces our task to showing that for $K,R=0,1$ and $k,r=1,2$ we have
\begin{align} \nonumber
& \bigg( \st+\svp \bigg)^{Kk} \bigg( \ct + \cvp \bigg)^{Rr}
\iint
\frac{\piK \, \piR}{(t^2 + 1 - u \st \svp - v \ct \cvp )^{\a + \b + 2 + (j + Kk +Rr)\slash 2  }} \\ \label{reduc}
    & \quad \lesssim
\Big( t^2+ \t^2+\v^2 \Big)^{-\a-1/2}
    \Big( t^2 + (\pi-\t)^2 + (\pi-\v)^2 \Big)^{-\b-1/2}
\frac{1}{\big[ t^2+(\t-\v)^2 \big]^{1+j/2}},
\end{align}
uniformly in $t \le 1$ and $\t,\v \in (0,\pi)$.
Applying Lemma~\ref{lem:intest} twice, first to the integral against
$\piR$ with the parameters
$\nu = (\b +1)R - (1-R)/2$, $\kappa = 0$, $\gamma = \a + \b + 2 + (j + Kk + Rr)/2$,
$A=t^2 + 1 - u \st \svp$, $B = \ct \cvp$, and then to the resulting integral against $\piK$ with
$\nu = (\a+1)K - (1-K)/2$, $\kappa = (\b+1)R + R/2$,
$\gamma = (\b+1)(1-R) + \a+1 - R/2 + (j + Kk + Rr)/2$,
$D = t^2 + 1$, $A = t^2 + 1 - \ct \cvp$, $B = \st \svp$,
we arrive at the relation
\begin{align*}
    & \iint
\frac{\piK \, \piR}{(t^2 + 1 - u \st \svp - v \ct \cvp )^{\a + \b + 2 + (j + Kk +Rr)\slash 2  }}\\
     & \simeq
\Big( t^2+ 1 - \ct \cvp \Big)^{-(\a+1)K -K/2}
    \Big( t^2 + 1 - \st \svp \Big)^{-(\b+1)R -R/2} \\
     & \qquad \times
\frac{1}{\big( t^2 + 1 - \st \svp - \ct \cvp \big)^
{ (\b + 1)(1-R) + (\a + 1)(1-K) - K/2 - R/2 + (j + Kk +Rr)/2 }}.
\end{align*}
Using \eqref{elem} we obtain that the left-hand side of \eqref{reduc} is comparable with
\begin{align*}
    & \bigg( \frac{\t^2 + \v^2}{t^2 + \t^2 + \v^2} \bigg)^{Kk/2}
\bigg( \frac{(\pi - \t)^2 + (\pi - \v)^2}{t^2 + (\pi - \t)^2 + (\pi - \v)^2} \bigg)^{Rr/2}\\
 & \times  \bigg( \frac{t^2 + (\t - \v)^2 }{t^2 + \t^2 + \v^2} \bigg)^
 {-(\a + 1/2)(1-K) + K(1-k/2)}
\bigg( \frac{t^2 + (\t - \v)^2 }{t^2 + (\pi - \t)^2 + (\pi - \v)^2} \bigg)^
{-(\b + 1/2)(1-R) + R(1-r/2)} \\
 & \times 
    \Big( t^2+ \t^2+\v^2 \Big)^{-\a-1/2} \Big( t^2 + (\pi-\t)^2 + (\pi-\v)^2 \Big)^{-\b-1/2}
\frac{1}{\big[ t^2+(\t-\v)^2 \big]^{1+j/2}}.
\end{align*}
This leads directly to the desired bound.
\end{proof}

\begin{proof}[Proof of Lemma \ref{lem:Step1}.]
We consider the case $\tau_{\ab} \ne 0$, leaving the opposite one to the reader.
Proceeding in a similar way as in the chain of identities \eqref{ch1}, i.e.\ using Proposition~\ref{prop:FJdecay}, \eqref{d5} and Fubini's theorem, we get
\begin{equation*}
\partial_{\t}^j \big(\J^{\ab}\big)^{-N/2} f (\t)
=
\frac{1}{\Gamma(N)} \int_0^\infty \int_0^{\pi} \partial_{\t}^j {H}_t^{\ab}(\t,\v) f (\v) \, d\mu_{\ab}(\v) \,
    t^{N-1}\, dt,
\qquad \t \in (0,\pi).
\end{equation*}
By Lemma~\ref{lem:diffK} it suffices now to check that the order of integration on the right-hand side above can be exchanged.
This will follow from Fubini's theorem once we ensure that
$$
\int_0^\infty \int_0^{\pi} | \partial_{\t}^j H_t^{\ab}(\t,\v)  f (\v) | \, d\mu_{\ab}(\v) \,
    t^{N-1}\, dt < \infty.
$$

Notice that since $f \in C_c^\infty(0,\pi)$, $d\mu_{\ab}(\v)$ is comparable with $d\v$ on $\support f$. Taking into account Lemma~\ref{lem:thetaJP} and applying
the Fubini-Tonelli theorem we get
\begin{align*}
\int_0^\infty \int_0^{\pi} | \partial_{\t}^j H_t^{\ab}(\t,\v)  f (\v) | \, d\mu_{\ab}(\v) \,
    t^{N-1}\, dt
    & \lesssim
\int_0^{\pi} \bigg( \int_0^1 \frac{t^N \, dt}{\big[ t^2+(\t-\v)^2 \big]^{1+j/2}}
+
\int_1^\infty e^{-ct} t^{N-1} \, dt
\bigg) \, d\v \\
    & \lesssim
\int_0^{\pi} \bigg( \int_0^1 \frac{t \, dt}{ t^2+(\t-\v)^2 }
+ 1 \bigg) \, d\v \\
    & \simeq
\int_0^{\pi} \log \bigg( 1 + \frac{1}{(\t -\v)^2 } \bigg) \, d\v
    < \infty.
\end{align*}
The conclusion follows. 
\end{proof}

\section{Proof of Lemma~\ref{lem:weightRiesz}}\label{sec:weightRiesz}

In this section we gather various technical results needed to 
conclude finally Lemma~\ref{lem:weightRiesz}. We start with a local refinement of Lemma~\ref{lem:tNSS} in case when $j$ is odd.

\begin{lem}\label{lem:oddthetaH}
Let $m \ge 1$ be fixed.
Then the following estimates hold uniformly in $t \le 1$ and
$\t,\v \in (0,\pi)$.
\begin{itemize}
\item[(i)]
If $\a,\b \ge -1\slash 2$, then
\begin{align*}
\big| \partial_\t^{2m-1}
H_{t}^{\ab}(\t,\v)
\big|
    \lesssim
t
\iint
\frac{|\partial_\t \q | \, \pia \, \pib}{(t^2 + \q)^{   \a + \b + 2 + m  }}.
\end{align*}
\item[(ii)]
If $-1 < \a < -1\slash 2 \le \b$, then
\begin{align*}
\big|
     \partial_\t^{2m-1}
H_{t}^{\ab}(\t,\vp)
\big|
    \lesssim
t
\sum_{K=0,1}
\iint
\frac{|\partial_\t \q  | \, \piK \, \pib}{(t^2 + \q)^{   \a + \b +  2 + m + K} }
+
t\iint
\frac{d\Pi_{\alpha + 1}(u) \, \pib}{(t^2 + \q)^{   \a + \b +  2 + m} }.
\end{align*}
\item[(iii)]
If $-1 < \b < -1\slash 2 \le \a$, then
\begin{align*}
\big|
    \partial_\t^{2m-1}
H_{t}^{\ab}(\t,\v)
\big|
    \lesssim
t
\sum_{R=0,1}
\iint
\frac{|\partial_\t \q  | \, \pia \, \piR}{(t^2 + \q)^{   \a + \b +  2 + m + R} }
+
t\iint
\frac{\pia \, d\Pi_{\beta + 1}(v)}{(t^2 + \q)^{   \a + \b +  2 + m} }.
\end{align*}
\item[(iv)]
If $-1 < \a,\b < -1\slash 2$, then
\begin{align*}
\big|
    \partial_\t^{2m-1}
H_{t}^{\ab}(\t,\v)
\big|
    &\lesssim
t
\sum_{K,R=0,1}
\iint
\bigg(
\frac{|\partial_\t \q  |}{(t^2 + \q)^{   \a + \b + 2 + m + K + R  }} \\
    & \quad  +
\chi_{ \{ K+R > 0 \} } \frac{1}{ (t^2 + \q)^{   \a + \b + 2 + m + K + R -1  } }
\bigg)\piK \, \piR.
\end{align*}
\end{itemize}
\end{lem}

\begin{proof}
Proceeding in a similar way as in the proof of Lemma~\ref{lem:tNSS} we reduce our task to showing that
\[
\big| \partial_u^K \partial_v^R \partial_\t^{2m-1} \Ptilde \big|
\lesssim
\frac{|\partial_\t \q  |}{(t^2 + \q)^{   \lam + m + K + R  }}
+
\chi_{ \{ K+R > 0 \} } \frac{1}{ (t^2 + \q)^{   \lam + m + K + R - 1 } },
\]
uniformly in $t \le 1$, $\t,\v \in (0,\pi)$ and $u,v \in [-1,1]$, where
$\lam \in \mathbb{R}$, $m \ge 1$, $K,R \in \{ 0, 1 \}$ are fixed.
This bound in turn follows
by a careful analysis of the proof of \cite[(11)]{NSS}.

Proceeding as in \cite{NSS} and using in addition the fact that
$\q, \partial_u \q, \partial_v \q, \partial_u \partial_\t \q, \partial_v \partial_\t \q$ are bounded we obtain
\[
\big| \partial_u^K \partial_v^R \partial_\t^{2m-1} \Ptilde \big|
\lesssim
\sum \frac{1}{(t^2 + \q)^{ \lam +  \sum_i j_i + k_1 + r_1  }}
| \partial_\t \q |^{ \sum_{\text{odd} \, i} j_i - k_2 - r_2 },
\]
where the main summation runs over $j_i \ge 0$, $j_1 + \ldots + (2m-1)j_{2m-1} = 2m-1$,
$k_1 + k_2 \le K$, $r_1 + r_2 \le R$ such that the exponent of
$| \partial_\t \q |$ is non-negative.
We distinguish two cases depending on whether this exponent is strictly positive or $0$.

\noindent \textbf{Case 1:} $\sum_{\text{odd} \, i} j_i \ge k_2 + r_2 + 1$.
Using the estimate $| \partial_\t \q | \lesssim \sqrt{\q} \le \sqrt{t^2 + \q}$, see \eqref{est4.5} below, 
and then the bound
$\sum_i j_i - \frac{1}2 \sum_{\text{odd} \, i} j_i \le m - 1/2$,
cf.\ \cite[(12)]{NSS}, we infer that
\begin{align*}
\frac{1}{(t^2 + \q)^{  \lam + \sum_i j_i + k_1 + r_1  }}
| \partial_\t \q |^{ \sum_{\text{odd} \, i} j_i - k_2 - r_2 }
\lesssim
\frac{1}{(t^2 + \q)^{ \lam + m  + k_1 + r_1 + (k_2 + r_2)/2 }}
| \partial_\t \q |.
\end{align*}
Since $r_1 + r_2 / 2 \le R$ and $k_1 + k_2 / 2 \le K$, we get the claimed bound.

\noindent \textbf{Case 2:} $\sum_{\text{odd} \, i} j_i = k_2 + r_2 $.
Notice that $\sum_{\text{odd} \, i} j_i$ is odd and consequently $\sum_{\text{odd} \, i} j_i = k_2 + r_2 = 1 $. Clearly, this forces
\begin{align*}
\frac{1}{(t^2 + \q)^{  \lam + \sum_i j_i + k_1 + r_1  }}
&\le
\chi_{ \{ k_2 = 1, r_2 = 0 \} }
\frac{1}{(t^2 + \q)^{  \lam + \sum_i j_i + r_1  }}
+
\chi_{ \{ k_2 = 0, r_2 = 1 \} }
\frac{1}{(t^2 + \q)^{  \lam + \sum_i j_i + k_1  }} \\
&\lesssim
\chi_{ \{ K=1 \} } \frac{1}{ (t^2 + \q)^{   \lam + m + R  } }
+
\chi_{ \{ R=1 \} } \frac{1}{ (t^2 + \q)^{   \lam + m + K  } },
\end{align*}
where the last estimate follows from the inequality $\sum_i j_i \le m$.
The relevant bound again follows.

The proof of Lemma~\ref{lem:oddthetaH} is finished.
\end{proof}

The following result is a local improvement of \cite[Lemma 4.5]{NoSj}, which says that
\begin{equation}\label{est4.5}
| \partial_\t \q | \lesssim \sqrt{\q}, \qquad \t,\v \in (0,\pi), \quad u,v \in [-1,1].
\end{equation}

\begin{lem}\label{lem:derqcomp}
Let $K$ be a fixed compact subset of $(0,\pi)$. Then
\[
\partial_\t \q
=
\frac{1}{2} \sin \frac{\t-\v}{2} + Q, \qquad \t \in K, \quad \v \in (0,\pi), \quad u,v \in [-1,1],
\]
with $Q=Q(\t,\v,u,v)$ satisfying $|Q| \lesssim \q$ uniformly in $\t,\v,u,v$ as above.
\end{lem}

\begin{proof}
It can easily be seen that (cf.\ \cite[(22)]{NoSj})
\[
\q \simeq (\t - \v)^2 + (1-u) \t \v + (1-v) (\pi - \t) (\pi - \v),
\qquad \t,\v \in (0,\pi), \quad u,v \in [-1,1],
\]
which gives
\begin{equation}\label{11.1}
\q \simeq (\t - \v)^2 + (1-u) \v + (1-v) (\pi - \v),
\qquad \t \in K, \quad \v \in (0,\pi), \quad u,v \in [-1,1].
\end{equation}
Since (see \cite[(22)]{NoSj})
\begin{align}\label{iden1}
\partial_\t \q
= \frac{1}{2} \sin \frac{\t-\v}{2}
    + \frac{1}{2}\Big[(1-u) \ct \svp - (1-v)\st \cvp\Big],
\qquad \t,\v \in (0,\pi), \quad u,v \in [-1,1],
\end{align}
we arrive at the desired result.
\end{proof}

\begin{lem}\label{lem:intest2}
Let $\nu > -1/2$ and $\gamma \in \mathbb{R}$ be fixed. Then, uniformly in $a>0$ and $B \ge 0$,
\begin{equation*}
\int_0^1 \frac{s^{\nu-1/2} \, ds}{ (a+Bs)^{\gamma} }
\simeq
\begin{cases} (a + B)^{-\nu - 1/2} {a^{-\gamma + \nu +1/2}},
&  \gamma > \nu +1/2 \\
    { (a + B)^{-\nu - 1/2}} \big( 1 + \log^+(B/a) \big),
&  \gamma =  \nu +1/2 \\
    { (a + B)^{-\gamma}},
& \gamma < \nu + 1/2
    \end{cases}.
\end{equation*}
\end{lem}

In the proof of this lemma we will use the following relation, which is an immediate consequence of \cite[Lemma 3.2]{NR}.
Given $\xi \in \mathbb{R}$, we have
\begin{align}\label{rel}
\int_x^y t^{\xi} \, dt \simeq
\begin{cases}
|y-x| x^{\xi + 1}/y, & \xi < -1  \\
\log(y/x), & \xi = -1 \\
|y-x| y^\xi, & \xi > -1
\end{cases}, \qquad y \ge x > 0.
\end{align}
For further reference we state also the following estimates, which correspond to $\xi<-1$ above. For a fixed $\sigma > 0$ we have
\begin{align}\label{3.2}
\Big| \frac{1}{x^\sigma} - \frac{1}{y^\sigma} \Big|
\simeq
\frac{|x - y|}{(x \vee y ) (x \wedge y)^{\sigma}}
\le
|x-y| \Big( \frac{1}{x^{\sigma+1}} + \frac{1}{y^{\sigma+1}} \Big), \qquad x,y >0.
\end{align}
\begin{proof}[Proof of Lemma \ref{lem:intest2}.]
We may assume that $a < B/2$, since otherwise $a + Bs \simeq a$, $s \in (0,1]$, and the conclusion is straightforward.
Further, we split the region of integration onto $(0, a/B)$ and $(a/B, 1)$ denoting the corresponding integrals by $I_0$ and $I_1$, respectively.
We consider $I_0$ and $I_1$ separately.
The treatment of $I_0$ is trivial because $a + Bs \simeq a$, $s \in (0, a/B)$,
which forces $I_0 \simeq {a^{-\gamma + \nu +1/2} B^{-\nu - 1/2}}$.
Since $a + Bs \simeq Bs$ for $s \in (a/B, 1)$, an application of \eqref{rel} with $\xi = \nu - 1/2 - \gamma$ gives
\begin{equation*}
I_1
\simeq
\frac{1}{B^\gamma}
\int_{a/B}^1 s^{\nu - 1/2 - \gamma} \, ds
\simeq
\frac{1}{B^\gamma}
\begin{cases} ({B}/a)^{\gamma -\nu - 1/2},
&  \gamma > \nu +1/2  \\
\log(B/a),
&  \gamma = \nu +1/2  \\
1,
& \gamma < \nu +1/2
    \end{cases}.
\end{equation*}
Comparing $I_0$ with $I_1$ we get the required relation.
\end{proof}

We state an easy consequence of Lemma~\ref{lem:intest2}, which is a generalization of the case $\kappa = 0$ in Lemma~\ref{lem:intest}.

\begin{cor}\label{cor:intest3}
Let $\nu \ge -1/2$ and $\gamma > 0$ be fixed. Then, uniformly in $A > B \ge 0$,
\begin{equation*}
\int \frac{d\Pi_{\nu}(s)}{ (A - Bs)^{\gamma} }
\simeq
\begin{cases} A^{-\nu - 1/2} {(A - B)^{-\gamma + \nu +1/2}},
&  \gamma > \nu + 1/2  \\
{ A^{-\nu - 1/2}} \big[ 1 + \log^+\big(B/(A-B)\big) \big],
&  \gamma = \nu + 1/2  \\
{ A^{-\gamma}},
& \gamma < \nu + 1/2
    \end{cases}.
\end{equation*}
\end{cor}

\begin{proof}
The case $\nu = -1/2$ is trivial, so we may assume that $\nu > -1/2$.
Then 
\[
\int \frac{d\Pi_{\nu}(s)}{ (A - Bs)^{\gamma} }
\simeq
\int_0^1 \frac{(1 - s)^{\nu-1/2} \, ds}{ (A - Bs)^{\gamma} }
=
\int_0^1 \frac{s^{\nu-1/2} \, ds}{ (A - B + Bs)^{\gamma} }.
\]
Now the conclusion is a direct consequence of Lemma~\ref{lem:intest2} specified to $a = A - B$.
\end{proof}

\begin{lem}\label{lem:intqest} Let $\ab > -1$ and $\eps>0$ be fixed. Assume that
$\xi, \xi_1,\xi_2,\kappa_1,\kappa_2 \ge 0$ are fixed and such that
$\a+\xi_1+\kappa_1, \, \b+\xi_2+\kappa_2 \ge -1/2$. Further, let
$K$ be a fixed compact subset of $(0,\pi)$. Then
\begin{equation}\label{bound6.5}
\iint \frac{d\Pi_{\a+\xi_1+\kappa_1}(u) \, d\Pi_{\b+\xi_2+\kappa_2}(v) }{ \q ^{\a+\b+\xi_1+\xi_2+\xi+1\slash 2}}
     \lesssim  1
    + \chi_{\{\xi=1/2\}} \bigg( \frac{1}{|\t - \v|^2} \bigg)^\eps
     + \chi_{\{\xi>1/2\}} \bigg( \frac{1}{|\t - \v|^2} \bigg)^{\xi-1/2},
\end{equation}
uniformly in $\t \in K$, $\v \in (0,\pi)$, $\t \ne \vp$.
\end{lem}

To show this we need the following elementary estimate.
For each $\rho > 0$ fixed, 
\begin{equation}\label{logest}
\log^+x \lesssim x^{\rho}, \qquad x>0.
\end{equation}

\begin{proof}[Proof of Lemma \ref{lem:intqest}.]
    By the boundedness of $\q$ and the finiteness of $d\Pi_{\nu}$, $\nu \ge -1/2$, we obtain
    \begin{align*}
        \iint \frac{d\Pi_{\a+\xi_1+\kappa_1}(u) \, d\Pi_{\b+\xi_2+\kappa_2}(v) }{ \q ^{\a+\b+\xi_1+\xi_2+\xi+1\slash 2}}
            \lesssim 1 + \chi_{\{\tilde{\a} + \tilde{\b} + \xi + 1/2>0\}}
                \iint \frac{d\Pi_{\tilde{\a}+\kappa_1}(u) \, d\Pi_{\tilde{\b}+\kappa_2}(v) }{ \q^{\tilde{\a} + \tilde{\b} + \xi + 1/2}},
    \end{align*}
    where $\tilde{\a}=\a+\xi_1$ and $\tilde{\b}=\b+\xi_2$.
    Next, we assume that $\tilde{\a}+\tilde{\b} + \xi +1/2>0$ and estimate the second term above.
    Applying Corollary~\ref{cor:intest3} with the parameters
    $\nu=\tilde{\a}+\kappa_1$,
    $\gamma=\tilde{\a}+\tilde{\b}+\xi+1/2$,
    $A=1-v\cos\frac{\t}{2} \cos\frac{\v}{2} \simeq 1 $
    and $B= \sin\frac{\t}{2} \sin\frac{\v}{2}$
    we get
    \begin{align*}
        \iint \frac{d\Pi_{\tilde{\a}+\kappa_1}(u) \, d\Pi_{\tilde{\b}+\kappa_2}(v) }{ \q^{\tilde{\a} + \tilde{\b} + \xi + 1/2}}
            &\lesssim 1+\int \bigg( \chi_{ \{\tilde{\b} + \xi > \kappa_1 \} } \frac{1}{(1 - \st \svp - v\ct \cvp)^{\tilde{\b} + \xi - \kappa_1}} \\
            & \qquad \qquad \qquad + \chi_{ \{\tilde{\b} + \xi = \kappa_1 \} }
                \log^+ \frac{\st \svp}{1 - \st \svp - v \ct \cvp} \bigg) d\Pi_{\tilde{\b}+\kappa_2}(v) \\
            &\equiv 1 + I_1 + I_2.
    \end{align*}
    
    We analyze $I_1$ and $I_2$ separately. Another application of Corollary~\ref{cor:intest3} taken with
    $\nu=\tilde{\b}+\kappa_2$,
    $\gamma=\tilde{\b}+\xi-\kappa_1$,
    $A=1-\st \svp \simeq 1 $
    and $B= \ct \cvp$ leads to
    \begin{align*}
        I_1
            & \lesssim 1
                + \chi_{\{\xi - 1/2 = \kappa_1 + \kappa_2\}} \log^+ \frac{1}{|\t - \v|^2}
                + \chi_{\{\xi - 1/2 > \kappa_1 + \kappa_2\}} \bigg( \frac{1}{|\t - \v|^2} \bigg)^{\xi-1/2-(\kappa_1+\kappa_2)},
    \end{align*}
    because $1- \st \svp - \ct \cvp \simeq |\t - \v|^2$. To see that $I_1$ is bounded by the
    right-hand side of \eqref{bound6.5} we take into account that $\kappa_1+\kappa_2 \geq 0$
    and make use of \eqref{logest} with $\rho=\eps$, when $\xi=1/2$, and $\rho=\xi-1/2$,
    when $\xi>1/2$.

    As for $I_2$, we write
    \begin{align*}
        I_2 \lesssim &
            \Big( \chi_{ \{\tilde{\b} + \xi = \kappa_1, \, \tilde{\b}+\kappa_2>-1/2 \} }
        + \chi_{ \{\tilde{\b} + \xi = \kappa_1, \, \tilde{\b}+\kappa_2 =-1/2 \} }
            \Big) 
\int \log^+ \frac{1}{1 - \st \svp - v \ct \cvp} d\Pi_{\tilde{\b}+\kappa_2}(v).
    \end{align*}
Since $d\Pi_{-1/2}$ is a simple atomic measure, the estimate related to the second term is straightforward.
Using \eqref{logest} with a certain $0<\rho<\tilde{\b}+\kappa_2+1/2$ to the integrand connected with the
first term above we see that the desired bound follows from Corollary~\ref{cor:intest3} specified to
$\nu=\tilde{\b}+\kappa_2$,
    $\gamma=\rho$, $A=1-\st \svp \simeq 1$ and $B=\ct \cvp$. 
\end{proof}

The next lemma will play a crucial role in the proof of Lemma~\ref{lem:weightRiesz}.

\begin{lem}\label{lem:thetaK}
Let $\ab > -1$, $m \ge 1$ and $0 \le j < 2m$. Assume that $K$ is a fixed compact subset of $(0,\pi)$. Then
\begin{align*}
| \partial_\t^j K_m^{\ab}(\t,\v) | \lesssim 1,
\qquad \t \in K, \quad \v \in (0,\pi), \quad \t \ne \v,
\end{align*}
where one should replace $K_m^{\ab}(\t,\v)$ by $\widetilde{K}^{\ab}_m(\t,\v)$ in case
$\tau_{\ab} = 0$
(actually, after this replacement the estimate is valid for all $\ab > -1$).
\end{lem}

\begin{proof}
We consider the case $\tau_{\ab} \ne 0$, leaving the opposite one to the reader.
An application of Lemmas~\ref{lem:diffK} and 
\ref{lem:thetaJP}
yields
\[
| \partial_\t^j K_m^{\ab}(\t,\v) |
\lesssim 1
+
\int_0^1 | \partial_\t^j H_t^{\ab}(\t,\v) | t^{2m-1} \, dt,
\qquad \t,\v \in (0,\pi), \quad \t \ne \v.
\]
Then for $j \le 2m -2$, with the aid of Lemma~\ref{lem:thetaJP}, we obtain
\[
| \partial_\t^j K_m^{\ab}(\t,\v) |
\lesssim 1
+
\int_0^1 \frac{ t^{2m} }{(t + |\t - \v|)^{2+j}} \, dt
\le 2,
\qquad \t \in K, \quad \v \in (0,\pi), \quad \t \ne \v. 
\]
Thus it remains to consider $j = 2m - 1$.

We focus on the case $-1 < \ab < -1/2$, which is the most involved one (see Lemma \ref{lem:oddthetaH});
the remaining cases are similar and hence left to the reader.
Using (iv) of Lemma~\ref{lem:oddthetaH} and then Lemma~\ref{lem:intest2} to the integral against $dt$ with the parameters
$\nu = 2m+1/2$, $\gamma = 2(\a + \b + 2 + m + W)$, $a = \sqrt{\q}$, $B=1$, where
$W \in \{K+R, K+R-1\}$, we infer that
\begin{align*}
| \partial_\t^{2m-1} K_m^{\ab}(\t,\v) |
    &\lesssim 1
+ \sum_{K,R = 0,1} \int_0^1 t^{2m} \iint \bigg(
\frac{|\partial_\t \q|}{(t^2 + \q)^{\a + \b + 2 + m + K + R}} \\
    & \qquad \qquad \qquad \qquad
+
\chi_{\{K+R>0\}}
\frac{1}{(t^2 + \q)^{\a + \b + 2 + m + K + R - 1}}
\bigg) \, \piK \, \piR \, dt \\
    & \lesssim
1 + \sum_{K,R = 0,1} \iint \bigg(
\frac{|\partial_\t \q|}{\q^{\a + \b + 3/2 + K + R}}
+
|\partial_\t \q| \log^+\frac{1}{\q}
+
\frac{1}{\q^{\a + \b + 1/2 + K + R}}
\\
    & \qquad \qquad \qquad \qquad
+
\chi_{\{K+R>0\}} \log^+\frac{1}{\q}
\bigg) \, \piK \, \piR,
\end{align*}
provided that $\t \ne \v$.
In view of \eqref{est4.5} and \eqref{logest}
the second term
under the last double integral is bounded.
Therefore an application of Lemma~\ref{lem:derqcomp} to the first, and \eqref{logest}
to the last term under the double integral above (with
$\rho = (\a + 1)K - (1-K)/2 + (\b + 1)R - (1-R)/2 + 1/2$ and $x=1/\q$)
implies
\begin{align*}
| \partial_\t^{2m-1} K_m^{\ab}(\t,\v) |
    &\lesssim
1 + \sum_{K,R = 0,1} \iint \bigg(
\frac{1}{\q^{\a + \b + 1/2 + K + R}}
+
\frac{|\t - \v|}{\q^{\a + \b + 3/2 + K + R}}
\\
    & \qquad \qquad \qquad \qquad \quad
+
\frac{1}{\q^{(\a + 1)K - (1-K)/2 + (\b + 1)R - (1-R)/2 + 1/2}}
\bigg) \, \piK \, \piR,
\end{align*}
for $\t \in K$, $\v \in (0,\pi)$, $\t \ne \v$.

The expression emerging from integration of the first two terms under the last double integral can be suitably
bounded by means of Lemma~\ref{lem:intqest} specified to $\xi_1 = K$, $\kappa_1 = (-\a -1/2)(1-K)$,
$\xi_2 = R$, $\kappa_2 = (-\b -1/2)(1-R)$; and $\xi=0$ or $\xi=1$, in the first or the second case, respectively.
The remaining term can also be treated by Lemma~\ref{lem:intqest}, this time applied with
$\xi=  \xi_1 = \xi_2 = \kappa_1 = \kappa_2 = 0$
and $\ab$ replaced by
$(\a + 1)K - (1-K)/2$ and $(\b + 1)R - (1-R)/2$.

This finishes the proof of Lemma~\ref{lem:thetaK}.
\end{proof}

We are now in a position to prove Lemma \ref{lem:weightRiesz}.

\begin{proof}[Proof of Lemma \ref{lem:weightRiesz}.]
We write $N=2m$, $m \ge 1$. From the definition of the Muckenhoupt class of $A^{\ab}_p$ weights we have that $1/w \in L^\infty(0,\pi)$ for $w \in A_1^{\ab}$ and
$w \in L^1(d\mu_{\ab})$ for $w \in A_p^{\ab}$, $1 \le p < \infty$. Therefore it is enough to show that for every fixed $\t \in (0,\pi)$ we have
\[
| R_{2m}^{\ab} (\t,\v) | \lesssim 1, \qquad \v \in (0,\pi), \quad \v \ne \t.
\]
We assume that $\tau_{\ab} \ne 0$; the opposite case can be treated in an analogous way. Let $\t \in (0,\pi)$ be fixed. By Corollary~\ref{cor:Rker} and the decomposition \eqref{dec1} we get
$$
R_{2m}^{\ab} (\t,\v) 
    = (-1)^m  \big(\J^{\ab}\big)_{\t}^m K_m^{\ab} (\t,\v)
 + \sum_{0 \le j < 2m} f_{m,j}(\t) \partial_{\t}^j K_m^{\ab} (\t,\v),
\qquad  \v \in (0,\pi), \quad \v \ne \t.
$$
Thus, in view of \eqref{HarmK} and Lemma~\ref{lem:thetaK}, we get the desired conclusion.
\end{proof}

\begin{rem}
Item \emph{(i)} in Lemma~\ref{lem:weightRiesz} can also be proved by using Lemma~\ref{lem:tNSS} and the fact that
$A_p^{\ab} \subset A_{\infty}^{\ab}$, $1\le p <\infty$.
On the other hand, Lemma~\ref{lem:weightRiesz} \emph{(ii)} is more subtle and cannot be deduced directly by an application of Lemma~\ref{lem:tNSS}.
\end{rem}

\section{Proof of Lemma~\ref{lem:PV=0}} \label{sec:PV0}

We start with some preparatory results. To state them, and also for further use, we denote
\begin{align*}
\mu_{\ab}(\v) & = \Big( \sin \frac{\v}{2} \Big)^{2\a+1} \Big( \cos \frac{\v}{2}\Big)^{2\b+1}, 
\qquad \v \in (0,\pi), \\
d & = d(\t)  =  \t \wedge (\pi-\t), \qquad \t \in (0,\pi) \quad
\big(\textrm{distance from $\t$ to the boundary of $(0,\pi)$}\big).  
\end{align*}
Further, for $\t \in (0,\pi)$ fixed and $\v \in (\t-d/2,\t + d/2)$ and $u,v \in [-1,1]$, 
we introduce the abbreviations
\begin{displaymath}
\begin {array}{rclrrl}
 \q(2\t-\v) & = & q(\t,2\t-\v,u,v), \qquad &
\partial_\t \q(2\t-\v) &  = & \partial_\t q(\t,2\t-\v,u,v), \\
 \partial_u \q(2\t-\v) & = & \partial_u q(\t,2\t-\v,u,v), \qquad &
\partial_v \q(2\t-\v) & = & \partial_v q(\t,2\t-\v,u,v).
\end {array}
\end{displaymath}
Finally, we will use frequently, sometimes without mentioning, the estimates
\begin{equation}\label{comp1}
\begin{split}
\cht - 1 \simeq t^2, \qquad 
\sht \simeq t,
\qquad \qquad  \qquad  0< t \le 1, \\
\mu_{\ab} (\v) \simeq 
\mu_{\ab} (2\t - \v) \simeq 1, 
\qquad \textrm{ $\t \in (0,\pi)$ fixed, \quad $\v \in (\t-d/2,\t + d/2)$}.
\end{split}
\end{equation}

\begin{lem}\label{lem:---}
    Let $\ab >-1$, $\gamma > 0$ and  $\t \in (0,\pi)$ be fixed. Then the following estimates hold.
\begin{align*}
& \emph{(a)} \qquad
	\big| \mu_{\ab}(\v) - \mu_{\ab}(2\t-\v) \big|
        \lesssim 
|\t-\v|, \qquad \v \in (\t-d/2,\t + d/2).\\
& \emph{(b)} \qquad
	\big| \q - \q(2\t-\v) \big|
        \lesssim |\t-\v| \, \q 
				\simeq |\t-\v| \, \q (2\t-\v), 
\qquad \v \in (\t-d/2,\t + d/2), \quad u,v \in [-1,1].\\
& \emph{(c)} \qquad
	\bigg| \frac{\mu_{\ab}(\v)}{(\cht - 1 + \q)^{\gamma}} - \frac{\mu_{\ab}(2\t-\v)}{(\cht - 1 + \q(2\t-\v))^{\gamma}} \bigg| \\
& \qquad \qquad \qquad  \lesssim
\frac{|\t - \v|}{(t^2 + \q)^{\gamma}}
            + 
\frac{|\t - \v|}{(t^2 + \q(2\t - \v))^{\gamma}},
\qquad t\le 1, \quad \v \in (\t-d/2,\t + d/2), \quad u,v \in [-1,1].
\end{align*}
\end{lem}

\begin{proof}
   Item (a) is a straightforward consequence of the Mean Value Theorem and the inequalities
\[
0< \frac{\t - d/2}2
\le 
\frac{\v}{2}, \frac{2\t-\v}{2}
\le 
\frac{\t + d/2}2
 < \frac{\pi}{2},
 \qquad \v \in (\t-d/2,\t + d/2).
\]

Next we deal with (b).
The second relation there is an immediate consequence of \eqref{11.1}. To prove the first one we use
    the sum-to-product trigonometric formulas,
    $$\sin A - \sin B = 2 \sin \frac{A-B}{2} \cos \frac{A+B}{2} \quad \text{and} \quad
      \cos A - \cos B = -2 \sin \frac{A+B}{2} \sin \frac{A-B}{2},$$
    to get
    \begin{align*}
        \big| \q - \q(2\t-\v) \big|
            & =  \bigg| u \sin \frac{\t}{2} \Big( \sin \frac{2\t-\v}{2} - \sin \frac{\v}{2} \Big)
                    + v \cos \frac{\t}{2} \Big( \cos \frac{2\t-\v}{2} - \cos \frac{\v}{2} \Big)\bigg| \\
           & =  2 \st \ct \Big|  \sin \frac{\t-\v}{2} \Big|
                    |u - v|   \\
          &  \lesssim |\t-\v| \big[(1-u) + (1-v) \big]
            \lesssim |\t-\v| \, \q, \qquad \v \in (\t-d/2,\t + d/2), \quad u,v \in [-1,1],
    \end{align*}
    where the last inequality follows from \eqref{11.1}.

Finally, we justify item (c). 
Using the triangle inequality, already proved item (a), \eqref{3.2} and the estimates \eqref{comp1},
we see that
\begin{align*}
       &  \bigg| \frac{\mu_{\ab}(\v)}{(\cht - 1 + \q)^{\gamma}} - \frac{\mu_{\ab}(2\t-\v)}{(\cht - 1 + \q(2\t-\v))^{\gamma}} \bigg| \\
           & \qquad \leq \frac{|\mu_{\ab}(\v)-\mu_{\ab}(2\t-\v)|}{(\cht - 1 + \q)^{\gamma}}  
          + \mu_{\ab}(2\t-\v)
 \bigg| \frac{1}{(\cht - 1 + \q)^{\gamma}} - \frac{1}{(\cht - 1 + \q(2\t-\v))^{\gamma}}\bigg|  \\
        & \qquad \lesssim \frac{|\t - \v|}{(t^2 + \q)^{\gamma}}
            + |\q-\q(2\t-\v)| \bigg[\frac{1}{(t^2 + \q)^{\gamma + 1}} + \frac{1}{(t^2 + \q(2\t-\v))^{\gamma + 1}}\bigg].
    \end{align*}
Now the conclusion follows from just proved item (b).
\end{proof}

\begin{cor}\label{cor:---}
    Let $\ab >-1$, $\gamma > 0$ and  $\t \in (0,\pi)$ be fixed. Then
\begin{align*}
& \bigg| \frac{\partial_\t \q \, \mu_{\ab}(\v)}{(\cht - 1 + \q)^{\gamma}} + \frac{\partial_\t \q(2\t - \v) \, \mu_{\ab}(2\t-\v)}{(\cht - 1 + \q(2\t-\v))^{\gamma}} \bigg| \\
& \qquad  \lesssim
\frac{\q}{(t^2 + \q)^{\gamma}}
            + 
\frac{\q (2\t - \v)}{(t^2 + \q(2\t - \v))^{\gamma}},
\qquad t\le 1, \quad \v \in (\t-d/2,\t + d/2), \quad u,v \in [-1,1].
\end{align*}
\end{cor}

\begin{proof}
Applying Lemma~\ref{lem:derqcomp} and then using \eqref{comp1} together with the relation $\big| \sin \frac{\t-\v}{2} \big| \simeq |\t-\v|$, we see that the left-hand side in question is bounded by
\begin{align*}
& |\t - \v| 
\bigg| \frac{ \mu_{\ab}(\v)}{(\cht - 1 + \q)^{\gamma}} - 
\frac{\mu_{\ab}(2\t-\v)}{(\cht - 1 + \q(2\t-\v))^{\gamma}} \bigg| 
+
\frac{\q}{(t^2 + \q)^{\gamma}}
            + 
\frac{\q (2\t - \v)}{(t^2 + \q(2\t - \v))^{\gamma}},
\end{align*}
for $t\le 1$, $\v \in (\t-d/2,\t + d/2)$ and $u,v \in [-1,1]$.
Then the asserted estimate is a direct consequence of Lemma~\ref{lem:---} (c) and the bound $|\t - \v|^2 \lesssim \q$.
\end{proof}

Now we are ready to prove Lemma~\ref{lem:PV=0}.

\begin{proof}[Proof of Lemma~\ref{lem:PV=0}]
Fix $\t \in (0,\pi)$. 
Using Lemma~\ref{lem:thetaJP} we see that 
for every $\eps>0$ we have
    \begin{align}\label{est1}
          \int_{0, \, |\v - \t|>\eps}^{\pi} \int_0^\infty \big|\partial_{\t} {H}_t^{\ab}(\t,\v)\big| \, dt \, d\mu_{\ab}(\v)
            < \infty.
    \end{align}
Consequently, by Fubini's theorem (see \eqref{d4}),
    $$\lim_{\eps \to 0} \int_{0, \, |\v - \t|>\eps}^\pi R_1^{\ab}(\t,\v) \, d\mu_{\ab}(\v)
        = \lim_{\eps \to 0} \int_0^\infty \int_{0, \, |\v - \t|>\eps}^{\pi} \partial_{\t} {H}_t^{\ab}(\t,\v) \,  d\mu_{\ab}(\v) \, dt.$$
    Hence, in view of \eqref{zeroint}, 
our task is reduced to showing that we can pass with the limit under the first integral in the right-hand side above.
To prove that this is indeed legitimate we will use the dominated convergence theorem.
Taking into account the identity
    $$
\int_{\t+\eps}^{\t+d/2} \partial_{\t} {H}_t^{\ab}(\t,\v) \, d\mu_{\ab}(\v)
        = \int_{\t-d/2}^{\t-\eps} \partial_{\t} {H}_t^{\ab}(\t,2\t-\v) \mu_{\ab}(2\t-\v) \, d\v, \qquad t>0, \quad 0 < \eps < d/2,
$$
which is a consequence of a simple change of variable,
it is sufficient to verify that
\[
\int_1^\infty \int_{0}^{\pi} \big| \partial_{\t} {H}_t^{\ab}(\t,\v) \big| \,  d\mu_{\ab}(\v) \, dt
+
\int_0^1  \int_{0, \, |\v - \t|> d/2}^\pi
 \big| \partial_{\t} {H}_t^{\ab}(\t,\v) \big| \,  d\mu_{\ab}(\v) \, dt
< \infty
\] 
and 
\begin{equation}\label{goal2}
        \int_0^1 \int_{\t-d/2}^{\t} \big| \partial_{\t} {H}_t^{\ab}(\t,\v)\mu_{\ab}(\v) + \partial_{\t} {H}_t^{\ab}(\t,2\t-\v) \mu_{\ab}(2\t-\v) \big| \, d\v \, dt
            < \infty.
    \end{equation}
Finiteness of the first two double integrals follows from Lemma~\ref{lem:thetaJP} and \eqref{est1} (with $\eps=d/2$), respectively.

Showing \eqref{goal2} is much more involved 
since there are some important cancellations between the two terms inside the absolute value.
To verify \eqref{goal2} it is convenient to distinguish $4$ cases emerging from different integral representations for the Jacobi-Poisson kernel ${H}_t^{\ab}(\t,\v)$ derived in \cite[Proposition 2.3]{NSS}, see also \cite[(1)]{NSS}.

    \noindent \textbf{Case 1:} $\ab \geq -1/2$. 
Differentiating \cite[Proposition 2.3 (i)]{NSS} with respect to $\t$, see \cite[(1)]{NSS}, we obtain
    \begin{align}\label{qderi}
		\partial_{\t} {H}_t^{\ab}(\t,\v)
        = C_{\ab} \sht \iint \frac{\partial_\t \q}{(\cht - 1 + \q)^{\a+\b+3}} \, \pia \, \pib, \qquad t>0, \quad \v \in (0,\pi);
				\end{align}
				here and later on in Cases 2-4 passing with the differentiation in $\t$ under the double integral is justified by means of the dominated convergence theorem, see the comment in the proof of \cite[Corollary 3.5]{NSS}.
Taking into account Corollary~\ref{cor:---} (with $\gamma = \a + \b + 3$) we see that the integral in \eqref{goal2} is controlled by
\begin{equation*}
         I=\int_0^1 \int_{\t-d/2}^{\t+d/2} \iint \frac{t \, \q}{(t^2 + \q)^{\a+\b+3}} \, \pia \, \pib  \, d\v \, dt.
    \end{equation*}
    By means of
    Lemma~\ref{lem:intest2} (specified to $\nu=3/2$, $\gamma=2(\a+\b+3)$, $a=\sqrt{\q}$ and $B=1$)
    and Lemma~\ref{lem:intqest} (taken with $\xi_1=\xi_2=\kappa_1=\kappa_2=0$, $\xi=1/2$, $\eps=1/4$) we get
    \begin{align*}
        I & \simeq  \int_{\t-d/2}^{\t+d/2} \iint \q \int_0^1 \frac{t\, dt}{(\sqrt{\q}+t)^{2(\a+\b+3)}}  \, \pia \, \pib  \, d\v \\
            &  \simeq  \int_{\t-d/2}^{\t+d/2} \iint 
\frac{\pia \, \pib}{\q^{\a+\b+1}} \, d\v
              \lesssim \int_{\t-d/2}^{\t+d/2} 
\frac{d\v}{|\t-\v|^{1/2}}
              < \infty.
    \end{align*}
This finishes proving \eqref{goal2} for $\ab \geq -1/2$.

    \noindent \textbf{Case 2:} $-1 < \a < -1/2 \leq \b$. This time  \cite[Proposition 2.3 (ii)]{NSS} leads to
\begin{equation}
\begin{split}
        \partial_{\t} {H}_t^{\ab}(\t,\v)
            = &  \; C_{\ab}^1 \sht \iint \frac{\partial_\t \q \, \partial_u \q}{(\cht - 1 + \q)^{\a+\b+4}} \, \Pi_\a(u)du \, \pib \\ \label{qderii}
            & + C_{\ab}^2 \sht \iint \frac{\partial_\t \partial_u \q }{(\cht - 1 + \q)^{\a+\b+3}} \, \Pi_\a(u)du \, \pib \\ 
            & + C_{\ab}^3 \sht \iint \frac{\partial_\t \q }{(\cht - 1 + \q)^{\a+\b+3}} \, d\Pi_{-1/2}(u) \, \pib, 
\quad \, \, \, \, t>0, \quad \v \in (0,\pi).
\end{split}
\end{equation}
Using the bounds $|\partial_\t \partial_u \q| \lesssim 1$,
\eqref{est4.5}
and the comparability \eqref{smallab},
the expression
    in \eqref{goal2} is controlled by
    \begin{align*}
 &  \int_0^1 \int_{\t-d/2}^{\t} t  \iint \bigg|\frac{\partial_\t \q \, \partial_u \q \, \mu_{\ab}(\v)}{(\cht - 1 + \q)^{\a+\b+4}}
                    + \frac{\partial_\t \q(2\t-\v) \, \partial_u \q(2\t-\v) \, \mu_{\ab}(2\t-\v)}{(\cht - 1 + \q(2\t-\v))^{\a+\b+4}}\bigg| \, d\Pi_{\a+1}(u) \, \pib \, d\v \, dt \\
 & \qquad + 
\int_0^1 \int_{\t-d/2}^{\t+d/2}  \iint \frac{t}{(t^2 + \q)^{\a+\b+3}} \, d\Pi_{\a+1}(u) \, \pib \, d\v \, dt \\
  & \qquad +
 \int_0^1 \int_{\t-d/2}^{\t+d/2}  \iint 
\frac{t \, \sqrt{\q} }{(t^2 + \q)^{\a+\b+3}} 
\, d\Pi_{-1/2}(u) \, \pib \, d\v  \, dt
\equiv  I_1 + I_2 + I_3.
\end{align*}
Observe that 
$$
\partial_u  \q \, \mu_{\ab}(\v)
        = - \st \, \mu_{\a+1/2,\b}(\v), 
\qquad \v \in (0,\pi), \quad u,v \in [-1,1],
$$
which, with the aid of Corollary~\ref{cor:---} (taken with $\a$ replaced by $\a+1/2$ and $\gamma = \a + \b + 4$), yields
\[
I_1 \lesssim
\int_0^1 \int_{\t-d/2}^{\t+d/2}  \iint \frac{t \, \q}{(t^2 + \q)^{\a+\b+4}} 
\, d\Pi_{\a+1}(u) \, \pib \, d\v \, dt
\le 
I_2.
\]
Therefore it is enough to estimate $I_2 + I_3$. Combining Lemma~\ref{lem:intest2} (applied to the integrals with respect to $t$ and
specified to $\nu=3/2$, $\gamma=2(\a+\b+3)$, $a=\sqrt{\q}$, $B=1$) with Lemma~\ref{lem:intqest} 
(taken with $\kappa_1=\xi_2=\kappa_2=0$ and
$\xi_1 = - \a - 1/2$, $\xi= \a + 3/2 > 1/2$ if $K=0$ and $\xi_1=1$, $\xi=1/2$,
$\eps= \a + 1$ if $K=1$) we obtain
\begin{align*}
I_2 + I_3 \simeq
\sum_{K=0,1} \int_{\t-d/2}^{\t+d/2}  \iint \frac{1}{\q^{\a+\b+3/2 +K/2}} 
\, \piK \, \pib \, d\v 
\lesssim
\int_{\t-d/2}^{\t+d/2} \frac{d\v}{|\t-\v|^{2(\a+1)}} 
< \infty,
\end{align*}
which gives \eqref{goal2} in case $-1 < \a < -1/2 \leq \b$.

    \noindent \textbf{Case 3:} $-1 < \b < -1/2 \leq \a$. 
This case is parallel to Case 2, details are left to the reader.

    \noindent \textbf{Case 4:} $-1 < \ab < -1/2$. Differentiating $H_t^{\ab}(\t,\v)$, see \cite[Proposition 2.3 (iv)]{NSS}, produces
\begin{equation}
\begin{split}
\partial_{\t} H_t^{\ab}(\t,\v)
            =\; & C_{\ab}^1 \sht \iint \frac{\partial_\t \q \, \partial_u \q \, \partial_v \q}{(\cht - 1 + \q)^{\a+\b+5}} \, \Pi_\a(u)du \, \Pi_\b(v)dv \\
            & + C_{\ab}^2 \sht \iint \frac{\partial_\t \partial_u \q \, \partial_v \q + \partial_\t \partial_v \q \, \partial_u \q}{(\cht - 1 + \q)^{\a+\b+4}} \, \Pi_\a(u)du \, \Pi_\b(v)dv \\ \label{qderiv}
            & + C_{\ab}^3 \sht \iint \frac{\partial_\t \q \, \partial_u \q}{(\cht - 1 + \q)^{\a+\b+4}} \, \Pi_\a(u)du \, d\Pi_{-1/2}(v) \\
            & + C_{\ab}^4 \sht \iint \frac{\partial_\t  \partial_u \q }{(\cht - 1 + \q)^{\a+\b+3}} \, \Pi_\a(u)du \, d\Pi_{-1/2}(v) \\
            & + C_{\ab}^5 \sht \iint \frac{\partial_\t \q \, \partial_v \q}{(\cht - 1 + \q)^{\a+\b+4}} \, d\Pi_{-1/2}(u) \, \Pi_\b(v)dv \\
            & + C_{\ab}^6 \sht \iint \frac{\partial_\t \partial_v \q }{(\cht - 1 + \q)^{\a+\b+3}} \, d\Pi_{-1/2}(u) \, \Pi_\b(v)dv \\
            & + C_{\ab}^7 \sht \iint \frac{\partial_\t \q }{(\cht - 1 + \q)^{\a+\b+3}} \, d\Pi_{-1/2}(u) \, d\Pi_{-1/2}(v), 
\end{split}
\end{equation}
for $t>0$ and $\v \in (0,\pi)$. Using now the estimates 
   $|\partial_\t \partial_u \q|, |\partial_\t \partial_v \q|, 
|\partial_u \q|, |\partial_v \q| \lesssim 1$, \eqref{est4.5} and \eqref{smallab}, we see that
the left-hand side of \eqref{goal2} is bounded by 
\begin{align*}
	& \int_0^1 \int_{\t-d/2}^{\t} t  \iint \bigg|\frac{\partial_\t \q \, \partial_u \q \, \partial_v \q \, \mu_{\ab}(\v)}{(\cht - 1 + \q)^{\a+\b+5}}
                + \frac{\partial_\t \q(2\t-\v) \, \partial_u \q(2\t-\v) \, \partial_v \q(2\t-\v) \, \mu_{\ab}(2\t-\v)}{(\cht - 1 + \q(2\t-\v))^{\a+\b+5}}\bigg|  \\
	             & \qquad \qquad \qquad \quad \qquad \qquad \qquad \quad \times 
d\Pi_{\a+1}(u) \, d\Pi_{\b+1}(v) \, d\v \, dt \\
	& \qquad + 
\sum_{\substack{K,R = 0,1 \\ K+R > 0 }} 
\int_0^1 \int_{\t-d/2}^{\t+d/2}  \iint \frac{t}{(t^2 + \q)^{\a+\b+2 + K + R}} 
\, \piK \, \piR \, d\v \, dt \\
& \qquad + 
\sum_{\substack{K,R = 0,1 \\ K+R \le 1 }} 
\int_0^1 \int_{\t-d/2}^{\t+d/2}  \iint \frac{t \, \sqrt{\q}}
{(t^2 + \q)^{\a+\b+3 + K + R}} 
\, \piK \, \piR \, d\v \, dt 
\equiv I_1 + I_2 + I_3.
\end{align*}

We first deal with $I_1$. It is easy to see that
$$
\partial_u  \q  \, \partial_v  \q \, \mu_{\ab}(\v)
        = \st \, \ct \, \mu_{\a+1/2,\b+1/2}(\v), 
\qquad \v \in (0,\pi), \quad u,v \in [-1,1],
$$
which together with Corollary~\ref{cor:---} (with $\ab$ replaced by $\a+1/2$, $\b + 1/2$, respectively, and $\gamma = \a + \b + 5$) leads to
\[
I_1 \lesssim
\int_0^1 \int_{\t-d/2}^{\t+d/2}  \iint \frac{t \, \q}{(t^2 + \q)^{\a + \b + 5}} 
\, \piaa \, \pibb \, d\v \, dt \le I_2.
\]
Therefore it suffices to show that $I_2$ and $I_3$ are finite.

Applying Lemma~\ref{lem:intest2} 
(choosing $\nu=3/2$, $\gamma=2(\a+\b+ 2 + K+R+ W)$, $a=\sqrt{\q}$, $B=1$, 
where $W\in \{ 0 , 1 \}$)
and then Lemma~\ref{lem:intqest} 
(with $\xi_1 = (- \a - 1/2)(1 - K) + K$, $\xi_2= (- \b - 1/2)(1 - R) + R$, $\kappa_1=\kappa_2=0$, $\eps=1/4$ and 
$0 < \xi= 1/2 - (- \a - 1/2)(1 - K) - (- \b - 1/2)(1 - R) \le 1/2$ 
in case of $I_2$ and 
$0 < \xi=(\a + 1)(1-K)+(\b + 1)(1-R)+(K+R)/2 < 1$ in case of $I_3$) 
we arrive at the bound
\begin{align*}
I_2 + I_3 & \lesssim
\sum_{\substack{K,R = 0,1 \\ K+R > 0 }} 
\int_{\t-d/2}^{\t+d/2}  \iint \frac{\piK \, \piR}{\q^{\a+\b+1 + K + R}} 
\, d\v  
+ 
\sum_{\substack{K,R = 0,1 \\ K+R \le 1 }} 
\int_{\t-d/2}^{\t+d/2}  \iint \frac{\piK \, \piR}
{\q^{\a+\b+ 3/2 + K + R}} 
\, d\v \\
& \lesssim
\int_{\t-d/2}^{\t+d/2} \frac{d\v}{|\t-\v|^{1/2}} 
 + 
\sum_{\substack{K,R = 0,1 \\ K+R \le 1 }} 
\int_{\t-d/2}^{\t+d/2} 
\frac{d\v}{|\t-\v|^{2(\a + 1)(1-K)+2(\b + 1)(1-R)+K+R-1}}
< \infty.
\end{align*}
This finishes the reasoning for the case of $-1 < \ab < -1/2$. 

The proof of Lemma~\ref{lem:PV=0} is completed.
\end{proof}

\section{Appendix: proof of Proposition \ref{pro:unbound}} \label{sec:diff}

To begin with, we reduce the task to proving boundedness properties for simpler operators. 
Observe that 
for $f \in C_c^\infty (0,\pi)$
\begin{align*}
R_{2}^{\ab} f (\t) &= 
- f(\t)
- (\a+1/2)\cot\frac{\t}2 \, \partial_{\t} \big(\J^{\ab}\big)^{-1} f(\t)
+(\b+1/2)\tan\frac{\t}2 \, \partial_{\t} \big(\J^{\ab}\big)^{-1} f(\t) \\ 
& \quad + \tau_{\ab}^2 \big(\J^{\ab}\big)^{-1} f(\t), 
\qquad \t \in (0,\pi),
\end{align*}
where one should replace $f$ on the right-hand side by $\Pi_0 f$ when $\tau_{\ab} = 0$.
Since $(\J^{\ab})^{-1}$ is bounded on $L^1(d\mu_{\ab})$, it suffices to consider
\begin{align*}
T_1^{\ab} f (\t) 
= \cot\frac{\t}2 \, \partial_{\t} \big(\J^{\ab}\big)^{-1} f(\t)
\quad \textrm{and} \quad
T_2^{\ab} f (\t) 
= \tan\frac{\t}2 \, \partial_{\t} \big(\J^{\ab}\big)^{-1} f(\t),
\qquad \t \in (0,\pi),
\end{align*}
with appropriate modification when $\tau_{\ab} = 0$.
For symmetry reasons, we have 
$T_1^{\b,\a} f (\t) = - T_2^{\ab} \widetilde{f} (\pi - \t)$, where 
$\widetilde{f} (\t) = f (\pi - \t)$. Therefore proving Proposition~\ref{pro:unbound} reduces to showing the following.

\begin{lem}\label{lem:T1}
Let $\ab >-1$. Then $T_1^{\ab}$ is a bounded operator from $L^1(d\mu_{\ab})$ to $L^1((3\pi/4,\pi),d\mu_{\ab})$ and unbounded from $L^1(d\mu_{\ab})$ to $L^1((0,\pi/4),d\mu_{\ab})$.
\end{lem}

The key tool which allows us to obtain this result is the well-known Schur criterion.

\begin{lem}\label{lem:Schur}
Let $(X,\mu)$, $(Y,\nu)$ be $\sigma$-finite measure spaces and let $K(x,y)$ be a measurable 
complex-valued kernel defined on $X \times Y$. If there exists a constant $C>0$ such that
\begin{align}\label{Schurest}
\int_{X} |K(x,y)| \, d\mu(x) \le C, \qquad \textrm{a.a. } y \in Y,
\end{align}
then the integral operator $Tf(x) = \int_Y K(x,y) f(y) \, d\nu(y)$ is bounded from $L^1(Y,\nu)$ to $L^1(X,\mu)$.
Moreover, when $K$ is non-negative, the converse is true: boundedness of $T$ from $L^1(Y,\nu)$ to $L^1(X,\mu)$ implies \eqref{Schurest}.
\end{lem}

In the proof of Lemma~\ref{lem:T1} we will need also several technical results, which are gathered below. We begin with the following modification of Lemma~\ref{lem:intqest} 
(corresponding to $K=(0,\pi/4)$, which is not admitted there).

\begin{lem}\label{lem:K=0pi4} Let $\ab > -1$ be fixed. Assume that
$\xi, \xi_1,\xi_2,\kappa_1,\kappa_2 \ge 0$ are fixed and such that
$\a+\xi_1+\kappa_1, \, \b+\xi_2+\kappa_2 \ge -1/2$. Then
\begin{align*}
\iint \frac{d\Pi_{\a+\xi_1+\kappa_1}(u) \, d\Pi_{\b+\xi_2+\kappa_2}(v) }{ \q ^{\a+\b+\xi_1+\xi_2+\xi+1\slash 2}} 
     &  \lesssim  
 1 
+ \chi_{\{ \a + \xi_1 + \xi = 0 \}} \log^+ \frac{1}{|\t - \v|}
 + \bigg( \frac{1}{\t + \v} \bigg)^{ 2(\a + \xi_1 + \xi)} \\
& \quad \times \bigg[ 1 + \chi_{\{\xi=1/2\}} \log \frac{\t + \v}{|\t - \v|} 
+ \chi_{\{\xi>1/2\}} \bigg( \frac{\t + \v }{|\t - \v|} \bigg)^{2\xi - 1} \bigg], 
\end{align*}
uniformly in $\t \in (0,\pi/4)$, $\v \in (0,\pi)$, $\t \ne \vp$.
\end{lem}

\begin{proof}
Observe that without any loss of generality we may and do assume that 
$\xi_1 = \xi_2 = 0$. Further, since $\q$ is bounded and the measures 
$d\Pi_{\nu}$, $\nu \ge -1/2$, are finite, we have
\begin{align*}
\iint 
\frac{d\Pi_{\a+\kappa_1}(u) \, d\Pi_{\b+\kappa_2}(v) }
{ \q^{\a+\b+\xi+1\slash 2}}
            \lesssim 1 + \chi_{\{ \a + \b + \xi + 1/2 >0 \}}
                \iint \frac{d\Pi_{\a+\kappa_1}(u) \, d\Pi_{\b+\kappa_2}(v) }
{ \q^{\a + \b + \xi + 1/2}}.
\end{align*}
	Assuming that $\a + \b + \xi + 1/2 >0$ and
    applying Corollary~\ref{cor:intest3} to the integral against $d\Pi_{\b+\kappa_2}(v)$ with the parameters
    $\nu = \b+\kappa_2$,
    $\gamma = \a + \b+\xi+1/2$,
    $A=1- u \st \svp \simeq 1 $,
    $B= \ct \cvp$,
    we obtain
    \begin{align*}
        \iint \frac{d\Pi_{\a+\kappa_1}(u) \, d\Pi_{\b+\kappa_2}(v) }
{ \q^{\a+\b+\xi+1\slash 2} }
            & \lesssim 1 +
\int \bigg( 
\chi_{ \{ \a + \xi > \kappa_2 \} } 
\frac{1}{(1 - \ct \cvp - u\st \svp)^{\a + \xi - \kappa_2}} \\
            & \qquad \qquad \qquad + \chi_{ \{ \a + \xi = \kappa_2 \} }
                \log^+ \frac{1}{1 - \ct \cvp - u\st \svp} \bigg)\, d\Pi_{\a+\kappa_1}(u) \\
            &\equiv 1 + I_1 + I_2.
    \end{align*}
    
We now analyze $I_1$ and $I_2$ separately. To treat $I_1$ we apply again Corollary~\ref{cor:intest3} specified to
    $\nu=\a+\kappa_1$,
    $\gamma=\a+\xi-\kappa_2$,
    $A=1-\ct \cvp \simeq (\t + \v)^2 $,
    $B= \st \svp$, which leads to
    \begin{align*}
        I_1
            & \lesssim
\chi_{\{\xi - 1/2 < \kappa_1 + \kappa_2\}}
\bigg( \frac{1}{\t + \v} \bigg)^{ 2(\a + \xi - \kappa_2)} 
+ \chi_{\{\xi - 1/2 = \kappa_1 + \kappa_2\}}
\bigg( \frac{1}{\t + \v} \bigg)^{ 2(\a + \kappa_1 + 1/2)} 
\bigg( 1 + \log^+ \frac{\t \v}{|\t - \v|^2} \bigg) \\
& \quad
                + \chi_{\{\xi - 1/2 > \kappa_1 + \kappa_2\}} 
\bigg( \frac{1}{\t + \v} \bigg)^{ 2(\a + \kappa_1 + 1/2)} 
\frac{1}{|\t - \v|^{2(\xi - 1/2 - \kappa_1 - \kappa_2)}} \\
& \equiv J_1 + J_2 + J_3.
    \end{align*}
The required bound for $J_1$ is straightforward, so let us pass to $J_2$. Since the constraint
$\xi - 1/2 = \kappa_1 + \kappa_2$ implies $\xi \ge 1/2$ and 
$ \kappa_1 + 1/2 =  \xi -  \kappa_2$, one can easily check that the conclusion follows (when $\xi > 1/2$ it is convenient to use \eqref{logest}).
Considering $J_3$, in this case $\xi > 1/2$ and we get
\[
J_3 \le
\bigg( \frac{ |\t - \v| }{\t + \v} \bigg)^{2 \kappa_1}
|\t - \v|^{2 \kappa_2}
\bigg( \frac{1}{\t + \v} \bigg)^{ 2(\a + \xi)} 
\bigg( \frac{\t + \v }{|\t - \v|} \bigg)^{2\xi - 1},
\]
which leads to the desired estimate. This finishes the analysis related to $I_1$.

Finally, we deal with $I_2$. The case $\a + \xi = 0$ is straightforward, 
so from now on we assume that $\a + \xi = \kappa_2 > 0$.
To proceed it is convenient to distinguish two cases.

\noindent \textbf{Case 1:} $\a + \kappa_1 = -1/2$. 
Then in $I_2$ we have $\xi = 1/2 + \kappa_1 + \kappa_2 > 1/2$ and using \eqref{logest} with any $\rho$ satisfying $0 < \rho \le (2\xi - 1) \wedge (2\a + 2\xi)$ we infer that
\[
I_2 \lesssim
1 + \log^+ \frac{1}{|\t - \v|} 
\lesssim \frac{1}{|\t - \v|^{\rho}}
\lesssim
\bigg( \frac{1}{\t + \v} \bigg)^{ 2(\a + \xi)} 
\bigg( \frac{\t + \v }{|\t - \v|} \bigg)^{2\xi - 1}.
\]

\noindent \textbf{Case 2:} $\a + \kappa_1 > -1/2$. 
Applying \eqref{logest} with a certain $\rho > 0$ satisfying 
$\rho < (\a + \kappa_1 + 1/2) \wedge (\a + \xi)$ and then 
Corollary~\ref{cor:intest3} specified to
    $\nu=\a+\kappa_1$,
    $\gamma=\rho$,
    $A=1-\ct \cvp \simeq (\t + \v)^2 $,
    $B= \st \svp$,
we get
\[
I_2 \lesssim
\int \frac{d\Pi_{\a+\kappa_1}(u)}{(1 - \ct \cvp - u\st \svp)^{\rho}} 
\simeq
\bigg( \frac{1}{\t + \v} \bigg)^{ 2\rho} 
\lesssim
\bigg( \frac{1}{\t + \v} \bigg)^{ 2(\a + \xi)}.
\]
This finishes the proof of Lemma~\ref{lem:K=0pi4}.
\end{proof}

\begin{lem}\label{lem:**}
Let $\nu, \lambda \in \mathbb{R}$, $\kappa < 1$, $\gamma > -1$ be fixed and such that $\nu + \lambda \ge 0$ and $\gamma + 1 + \nu - \kappa \ge 0$.
Then, excluding the case when
$\nu + \lambda = \gamma + 1 + \nu - \kappa = 0$, we have
\[
\v^{\nu} \int_0^{\pi/4} 
\bigg( \frac{\v}{\t + \v} \bigg)^{\lambda} \frac{\t^{\gamma} \, d\t}{|\t - \v|^{\kappa}}  
\lesssim
1,\qquad \v \in (0,\pi).
\]
\end{lem}

\begin{proof}
Changing the variable of integration $\t = \v s$ and keeping in mind that $\kappa < 1$ and $\gamma > -1$,
we get
\begin{align*}
\v^{\nu} \int_0^{\pi/4} 
\bigg( \frac{\v}{\t + \v} \bigg)^{\lambda} \frac{\t^{\gamma} \, d\t}{|\t - \v|^{\kappa}} 
& = \v^{\gamma+1+\nu-\kappa} \int_{0}^{\pi/(4\v)} \bigg( \frac{1}{1+s}\bigg)^{\lambda} 
	\frac{s^{\gamma}\, ds}{|1-s|^{\kappa}} \\
& \simeq \v^{\gamma+1+\nu-\kappa} \bigg( 1 + 
	\chi_{\{\v < \pi/8\}}\int_2^{\pi/(4\v)} s^{\gamma - \lambda - \kappa}\, ds \bigg) \\
& \lesssim \v^{\gamma+1+\nu-\kappa} \Big( 1 + \chi_{\{\gamma-\lambda-\kappa+1=0\}} \log\frac{\pi}{\v}
		+ \v^{-\gamma + \lambda + \kappa -1} \Big).
\end{align*}
Clearly, the last expression is bounded uniformly in $\v \in (0,\pi)$, 
in view of the assumptions imposed on the parameters.
\end{proof}

The next result will be needed when dealing with the case $\a \ge -1/2$.

\begin{lem}\label{lem:-1/2}
Let $\a \ge -1/2$ and $\b > -1$ be fixed. 
Consider the kernel $K(\t,\v)$ defined on $(0,\pi/4) \times (0,\pi)$ in the following way.
\begin{itemize}
\item[(a)] For $ \b \ge -1/2$, 
\[
K(\t,\v) = 
\t^{-1} \int_0^1 t \sht \iint 
\frac{\sin\frac{\t - \v}2 + (1-u) \ct \svp}{(\cht - 1 + \q)^{\a + \b + 3}}
\, \pia \, \pib \, dt.
\] 
\item[(b)] For $-1 < \b < -1/2 $, 
\[
K(\t,\v) = 
\t^{-1} \int_0^1 t \sht \iint 
\frac{\sin\frac{\t - \v}2 + (1-u) \ct \svp}{(\cht - 1 + \q)^{\a + \b + 4}}
\, \pia \, \ob \, dt.
\]
\end{itemize}
Then we have
\[
\int_0^{(\pi/4)\wedge (2\v)}  |K(\t,\v)|
\, d\mu_{\ab}(\t) \lesssim 1, \qquad \v \in (0,\pi).
\]
\end{lem}

In the proofs of Lemmas~\ref{lem:-1/2} and \ref{lem:T1} we will use the fact that for each fixed $\nu > -1/2$ we have
\begin{equation}\label{largeab}
(1-u) \, d\Pi_{\nu} (u) \simeq d\Pi_{\nu + 1} (u), \qquad u \in (0,1).
\end{equation}

\begin{proof}[Proof of Lemma \ref{lem:-1/2}]
In the reasoning below we assume that $\t \le (\pi/4) \wedge (2 \v)$, if not stated otherwise.
Further, we define an auxiliary constant $\sigma = \sigma(\b)$ which is equal to $0$ if $\b \ge -1/2$ and $1$ if $-1 < \b < -1/2$.
We deal with items (a) and (b) simultaneously, but we consider the cases of
$\a > -1/2$ and $\a = -1/2$ separately.

\noindent \textbf{Case 1:} $\a > -1/2 $. 
By \eqref{smallab} we obtain
\[
|K(\t,\v)| 
\lesssim
\t^{-1} \int_0^1 t^2 \iint 
\frac{|\t - \v| + (1-u) \v }{(t^2 + \q)^{\a + \b + 3 + \sigma}}
\, \pia \, d\Pi_{\b + \sigma}(v) \, dt.
\]
Now applying Lemma~\ref{lem:intest2} 
specified to $\nu = 5/2$, $\gamma = 2(\a + \b + 3 + \sigma)$, $a=\sqrt{\q}$, 
$B = 1$, and then Corollary~\ref{cor:intest3}
to the integral against $d\Pi_{\b + \sigma}(v)$ with
$\nu = \b + \sigma$, $\gamma = \a + \b + 3/2 + \sigma$, 
$A = 1- u \st \svp \simeq 1$, $B=\ct \cvp$,
we see that
\begin{align*}
|K(\t,\v)|
& \lesssim 
\t^{-1}
\iint \frac{|\t - \v| + (1-u) \v }{\q^{\a+\b+3/2 + \sigma}} 
\, d\Pi_{\b + \sigma}(v) \, \pia \\ 
& \simeq
\t^{-1} 
\int \frac{|\t - \v| + (1-u) \v }{(1 - \ct \cvp - u \st \svp)^{\a+ 1}}  
\, \pia .
\end{align*}
To proceed, 
we split the region of integration in the last integral onto the intervals $[-1,0]$ and $[0,1]$, and 
denote the corresponding expressions by $I_{-1}$ and $I_1$, respectively. 
In order to finish the proof of Case 1 it suffices to show that
\[
I_{-1} + I_{1}
\lesssim 
 \t^{-1} \v^{-2\a - 1}.
\] 
Since 
\[
1 - \ct \cvp - u \st \svp \ge 1 - \ct \cvp \simeq (\t + \v)^2, \qquad 
u \in [-1,0], \quad \t,\v \in (0,\pi),
\]
the conclusion for $I_{-1}$ is trivial.
Using \eqref{largeab} and then Corollary~\ref{cor:intest3} twice 
(with $\nu = \a$ or $\nu = \a + 1$ and 
$\gamma = \a + 1$, $A = 1-  \ct \cvp \simeq (\t + \v)^2$, $B=\st \svp$) we get the required estimate for $I_1$.

\noindent \textbf{Case 2:} $\a = -1/2 $.
Computing the integral against $\piu$, applying the triangle inequality and then \eqref{3.2}, we see that
\begin{align*}
& \bigg| \int 
\frac{\sin\frac{\t - \v}2 + (1-u) \ct \svp}{(\cht - 1 + \q)^{\b + 5/2 + \sigma}}
\, \piu
\bigg| \\
& \qquad \simeq
\bigg| \frac{ \sin\frac{\t + \v}2 }{(\cht -1 + q(\t,\v,-1,v))^{\b + 5/2 + \sigma}} 
+ \frac{ \sin\frac{\t - \v}2 }{(\cht -1 + q(\t,\v,1,v))^{\b + 5/2 + \sigma}} 
\bigg| \\
& \qquad \le
\Big|\sin\frac{\t - \v}2 \Big| \bigg| \frac{1}{(\cht -1 + q(\t,\v,1,v))^{\b + 5/2 + \sigma}} 
- \frac{1}{(\cht -1 + q(\t,\v,-1,v))^{\b + 5/2 + \sigma}}  \bigg| \\
& \qquad \qquad +
\frac{1}{(\cht -1 + q(\t,\v,-1,v))^{\b + 5/2 + \sigma}} 
\Big| \sin\frac{\t - \v}2 + \sin\frac{\t + \v}2 \Big| \\
& \qquad \lesssim
 \frac{ |\t - \v| \t \v^{-1}}{ (t^2 + q(\t,\v,1,v) )^{\b + 5/2 + \sigma}}
+
\frac{\t}{ (t^2 + q(\t,\v,-1,v) )^{\b + 5/2 + \sigma}}.
\end{align*}
Combining this with \eqref{smallab}, Lemma~\ref{lem:intest2} (applied with $\nu = 5/2$, $\gamma = 2(\b + 5/2 + \sigma)$, $a=\sqrt{q(\t,\v,\pm 1,v)}$, $B = 1$) and Corollary~\ref{cor:intest3}
(with $\nu = \b + \sigma$,  
$\gamma = \b + 1 + \sigma$, $A = 1 \mp  \st \svp \simeq 1$, $B=\ct \cvp$) we obtain
\begin{align*}
|K(\t,\v)| 
 \lesssim
\int \bigg(
 \frac{ |\t - \v| \, \v^{-1}}{ q(\t,\v,1,v)^{\b+1 + \sigma}} 
+
\frac{1}{ q(\t,\v,-1,v)^{\b+1 + \sigma}}
\bigg) \, d\Pi_{\b + \sigma}(v)
 \simeq 
\v^{-1}, 
\end{align*}
which concludes Case 2, and thus the proof of Lemma~\ref{lem:-1/2}.
\end{proof}

\begin{lem}\label{lem:uintcanc}
Let $\a > -1$ and $\gamma > 0$ be fixed. 
\begin{itemize}
\item[(a)] If $ \a \ge -1/2$, then
\[
\bigg| \int \frac{u \, \pia}{(\cht - 1 + \q)^{\gamma}}
\bigg|
\lesssim
\frac{\t \v}{(\t + \v)^2}
\int \frac{\pia}{(t^2 + \q)^{\gamma}},
\qquad 0<t \le 1, \quad \t,\v \in (0,\pi), \quad v \in [-1,1].
\] 
\item[(b)] If $-1 < \a < -1/2 $, then
\[
\bigg| \int \frac{ \oa }{(\cht - 1 + \q)^{\gamma}}
\bigg|
\lesssim
\frac{\t \v}{(\t + \v)^2}
\int \frac{\piaa}{(t^2 + \q)^{\gamma}},
\qquad 0<t \le 1, \quad \t,\v \in (0,\pi), \quad v \in [-1,1].
\] 
\end{itemize}
\end{lem}

\begin{proof}
We will treat both cases simultaneously. Since the measures $u \, \pia$, $ \a \ge -1/2$, and $\oa$, $-1 < \a < -1/2$, are odd in $[-1,1]$, do not possess any atom at $0$ and have finite total variation, we get
\begin{align*}
\bigg| \int \frac{u \, \pia}{(\cht - 1 + \q)^{\gamma}}
\bigg|
& \le 
\int_{[0,1]} \bigg|
\frac{1}{(\cht - 1 + \q)^{\gamma}} 
- \frac{1}{(\cht - 1 + q(\t,\v,-u,v) )^{\gamma}}
\bigg| \, \pia, \\
\bigg| \int \frac{\oa}{(\cht - 1 + \q)^{\gamma}}
\bigg|
& \lesssim 
\int_{[0,1]} \bigg|
\frac{1}{(\cht - 1 + \q)^{\gamma}} 
- \frac{1}{(\cht - 1 + q(\t,\v,-u,v) )^{\gamma}}
\bigg| \, \piaa;
\end{align*}
to obtain the second estimate we used also \eqref{smallab}. Now the conclusion is an immediate consequence of \eqref{3.2} and the relations
\[
q(\t,\v,-u,v) \ge 1 - \ct \cvp \simeq (\t + \v)^2, 
\qquad \t,\v \in (0,\pi), \quad u \in [0,1], \quad v \in [-1,1].
\]
\end{proof}

Finally, in the proof of Lemma~\ref{lem:T1} we will frequently use the estimates, 
see \eqref{est4.5} and \eqref{q_for},
\begin{align}\label{est5}
|\partial_{\t} \q| \lesssim \sqrt{\q}, \qquad 
|\partial_u \q| \lesssim \t \v, \qquad 
|\partial_v \q| \lesssim 1, \qquad
|\partial_{\t} \partial_u \q| \lesssim \v, \qquad
|\partial_{\t} \partial_v \q| \lesssim \t, 
\end{align}
holding uniformly in $\t,\v \in (0,\pi)$ and $u, v \in [-1,1]$.

\begin{proof}[Proof of Lemma \ref{lem:T1}]
By Lemmas \ref{lem:Step1} and \ref{lem:diffK} we see that $T_1^{\ab}$ is an integral operator with the kernel
\[
K(\t,\v) = \cot\frac{\t}2 \int_0^\infty \partial_{\t} H_t^{\ab} (\t,\v) \, tdt,
\qquad \t,\v \in (0,\pi), \quad \t \ne \v;
\]
note that $\tau_{\ab} = 0$ is also included.
We first focus on the positive part of the lemma. Using sequently Lemma~\ref{lem:thetaJP}, Lemma~\ref{lem:intest2} (applied with $\nu = 5/2$, $\gamma = 3$, $a=|\t-\v|$ and $B=1$) and \eqref{logest} with $\rho=1/2$ we obtain
\begin{align*}
|K(\t,\v)|
& \lesssim
1 + (\pi - \t) \int_0^1 \frac{t^2\,dt}{(t + \pi - \t + \pi - \v)^{2\b+1} 
(t + |\t - \v|)^3 } 
\lesssim
1 + (\pi - \t)^{-2\b-1} \int_0^1 \frac{t^2\,dt}{(t + |\t - \v|)^3 } \\
& \lesssim
1 + (\pi - \t)^{-2\b-1} |\t-\v|^{-1/2}, 
\qquad \t \in (3\pi/4, \pi), \quad  \v \in (0,\pi).
\end{align*}
Now the desired conclusion is a direct consequence of Lemma~\ref{lem:Schur}.

We pass to proving the negative part. We split the region of integration in the definition of $K(\t,\v)$ onto $(0,1)$ and $(1,\infty)$ denoting the resulting expressions by $K_0 (\t,\v)$ and $K_{\infty}(\t,\v)$, respectively. We first show that the operator $T_\infty$ associated with the kernel $K_{\infty}(\t,\v)$ is $L^1(d\mu_{\ab})$-bounded.
Using the series definition of $H_t^{\ab} (\t,\v)$, see \eqref{d0}, and proceeding as in the proof of \cite[Lemma 3.8]{NSS}, we obtain
\[
|\partial_{\t} H_t^{\ab} (\t,\v)|
\lesssim
\sin \t \, e^{-ct}, \qquad t \ge 1, \quad \t,\v \in (0,\pi),
\]
for some $c=c_{\ab}>0$. This combined with Lemma~\ref{lem:Schur} gives us the desired property for $T_\infty$.

It remains to deal with the operator $T_0$ associated to the kernel $K_0(\t,\v)$.
We will show that $T_0$ is not bounded from $L^1(d\mu_{\ab})$ to $L^1((0,\pi/4),d\mu_{\ab})$.
This will finish the proof.
It is convenient to distinguish four cases depending on whether each of the parameters of type $\ab$ is less than $-1/2$ or not. 

\noindent \textbf{Case 1:} $\ab \ge -1/2$.
Using \eqref{qderi}
and the decomposition \eqref{iden1} we arrive at
\begin{align*}
K_0 (\t,\v) &= c \cot\frac{\t}2 \int_0^1 t \sht \iint 
\frac{\big[ \sin\frac{\t-\v}2  + (1-u) \ct \svp \big] - (1-v) \st \cvp}
{ (\cht - 1 + \q)^{\a + \b + 3} } \, \pia \, \pib \, dt \\
&\equiv c L_1(\t,\v) - c L_2(\t,\v), \qquad \t,\v \in (0,\pi), \quad \t \ne \v,
\end{align*}
with some non-zero constant $c=c_{\ab}$. 
We claim that $L_2 (\t,\v)$ produces a bounded operator from $L^1(d\mu_{\ab})$ to $L^1((0,\pi/4),d\mu_{\ab})$. Indeed, applying Lemma~\ref{lem:intest2} 
(specified to $\nu = 5/2$, $\gamma = 2(\a + \b + 3)$, $a=\sqrt{\q}$, 
$B = 1$), 
splitting the integration in $v$ into $[-1,0]$ and $[0,1]$, and then using the estimates
\begin{align}\label{est6}
1- u \st \svp \simeq 1, \qquad u\in [-1,1], 
\quad \t \in (0,\pi/4), \quad \v \in (0,\pi),
\end{align}
and \eqref{largeab},
we see that
\begin{align*}
|L_2(\t,\v)|
& \lesssim
\int_0^1 t^2 \iint \frac{(1-v)\, \pia \, \pib}{(t^2 + \q)^{\a+\b+3}} \, dt
\simeq
\iint \frac{(1-v)\, \pia \, \pib}{\q^{\a+\b+3/2}} \\
& \lesssim
1 + \int_{[0,1]} \int \frac{ \pia \, \pibb}{\q^{\a+\b+3/2}},
\qquad \t \in (0,\pi/4), \quad \v \in (0,\pi), \quad \t \ne \v;
\end{align*}
note that here $\b = -1/2$ is also included.
Now Lemma~\ref{lem:K=0pi4} (specified to 
$\xi_1 = \kappa_1 = \kappa_2 = \xi = 0$, $\xi_2 = 1$) leads to 
\[
|L_2(\t,\v)| 
\lesssim
(\t + \v)^{-2\a} + |\t -\v|^{-1/2}
\lesssim
(\t + \v)^{-2\a - 1} |\t -\v|^{-1/2}, 
\qquad \t \in (0,\pi/4), \quad \v \in (0,\pi), \quad \t \ne \v.
\]
This, in view of Lemma~\ref{lem:Schur} and Lemma~\ref{lem:**} 
(applied with $\gamma = \lambda = - \nu = 2\a + 1$, $\kappa = 1/2$), gives the asserted property for the operator connected with $L_2 (\t,\v)$.

We now focus on $L_1(\t,\v)$ and show that it produces an unbounded operator from $L^1(d\mu_{\ab})$
to $L^1((0,\pi/4),d\mu_{\ab})$. 
By combining Lemma~\ref{lem:-1/2} (a) with Lemma~\ref{lem:Schur} we know that $\chi_{ \{ \t \le 2\v \} } L_1(\t,\v)$ produces a bounded operator from $\Leb1$ to $\Lebr1$.
Since $\chi_{ \{ \t > 2\v \} } L_1(\t,\v) \ge 0$, in order to finish Case 1 it suffices to show that,
see Lemma \ref{lem:Schur},
\begin{align}\label{iden2}
\essup_{\v \in (0,\pi/8)} \int_{2\v}^{\pi/4} L_1(\t,\v) \, \t^{2\a+ 1} \, d\t
 = \infty.
\end{align}
Using Lemma~\ref{lem:intest2} (taken with $\nu = 5/2$, $\gamma = 2( \a + \b + 3)$, $a=\sqrt{\q}$, $B = 1$) and Corollary~\ref{cor:intest3} twice (first to the integral against $\pib$ with $\nu = \b$, $\gamma = \a + \b + 3/2$, $A = 1 - u \st \svp \simeq 1$, $B=\ct \cvp$, and then to the resulting integral against $\pia$ 
with $\nu = \a$, $\gamma = \a + 1$, $A = 1 - \ct \cvp \simeq (\t + \v)^2$, $B=\st \svp$) we obtain
\begin{align*}
\chi_{ \{ \pi/4 \ge \t > 2\v \} } L_1(\t,\v) 
& \gtrsim
\chi_{ \{ \pi/4 \ge \t > 2\v \} } \, \t^{-1}
\int_0^1 t^2 \iint 
 \frac{ \t - \v }{ (t^2 + \q)^{\a + \b+ 3}} 
\, \pia \, \pib \, dt \\
& \simeq 
\chi_{ \{ \pi/4 \ge \t > 2\v \} } \,
\iint \frac{\pib \, \pia}{  \q^{\a + \b+ 3/2}}  \\
& \simeq 
\chi_{ \{ \pi/4 \ge \t > 2\v \} } \, 
\int \frac{\pia}{  (1 - \ct \cvp - u \st \svp)^{\a + 1}}  
\simeq 
\chi_{ \{ \pi/4 \ge \t > 2\v \} } \, \t^{-2\a - 2}.
\end{align*}
This confirms \eqref{iden2} and finishes the reasoning justifying Lemma~\ref{lem:T1} for $\ab \ge -1/2$.

\noindent \textbf{Case 2:} $-1 < \a < -1/2 \le \b$.
Decompose $K_0(\t,\v) = L_1(\t,\v) + L_2(\t,\v) + L_3(\t,\v)$, where $L_j(\t,\v)$ corresponds to the term with constant $C^j_{\ab}$ in \eqref{qderii}, $j=1,2,3$ (observe that $C^j_{\ab} \ne 0$).
We first ensure that $L_1(\t,\v)$ and $L_2(\t,\v)$ are associated with bounded operators from $\Leb1$ to $\Lebr1$. 
Using Lemma~\ref{lem:uintcanc}~(b) (to $L_2(\t,\v)$ with $\gamma = \a + \b + 3$) and then \eqref{est5} and \eqref{smallab}, we get
\[
|L_1(\t,\v)| + |L_2(\t,\v)|
\lesssim
\int_0^1 t^2 \iint \bigg[ \frac{\v \, \sqrt{ \q } }{(t^2 + \q)^{\a+\b+4}}
+ \frac{\v^2}{(\t + \v)^2} 
\frac{1 }{(t^2 + \q)^{\a+\b+3}}
\bigg] 
\, \piaa \, \pib \, dt .
\]
Then an application of Lemma~\ref{lem:intest2} (specified to $\nu = 5/2$,  $a=\sqrt{\q}$, $B = 1$ and  $\gamma = 2( \a + \b + 4)$ or 
$\gamma = 2( \a + \b + 3)$) and then 
Lemma~\ref{lem:K=0pi4} (choosing $\xi_1 = 1$, 
$\kappa_1 = \xi_2 = \kappa_2 = 0$ and $\xi = 1/2$ or $\xi = 0$)
leads to
\begin{align*}
|L_1(\t,\v)| + |L_2(\t,\v)|
& \lesssim
\iint \bigg[ \frac{\v }{ \q^{\a+\b+2}}
+ \frac{\v^2}{(\t + \v)^2} 
\frac{1 }{\q^{\a+\b+3/2}}
\bigg] 
\, \piaa \, \pib \\
& \lesssim
\frac{\v}{(\t + \v)^{2\a + 3}}
\bigg( \frac{\t + \v}{|\t - \v|} \bigg)^{1/2}
+
\frac{\v^2}{ (\t + \v)^{2\a + 4} }
\lesssim
\frac{\v}{(\t + \v)^{2\a + 5/2}}
\frac{1}{|\t - \v|^{1/2}},
\end{align*}
provided that $\t \in (0,\pi/4)$, $\v \in (0,\pi)$, $\t \ne \v$.
This, with the aid of Lemma~\ref{lem:Schur} and Lemma~\ref{lem:**} (taken with $\gamma = 2\a + 1$, $\lambda = 2\a + 5/2$, $\nu = -2\a - 3/2$, $\kappa = 1/2$), finishes the analysis concerning $L_1(\t,\v)$ and $L_2(\t,\v)$.

It remains to check that $L_3(\t,\v)$ defines an unbounded operator from $L^1(d\mu_{\ab})$ to
$L^1((0,\pi/4),d\mu_{\ab})$.
Since
\begin{align}\label{iden4}
\partial_{\t} \q = -\frac{1}2 u \ct \svp + \frac{1}2 v \st \cvp, 
\end{align}
we consider the kernels $J_k(\t,\v)$, $k=-1,0,1$, given by
\begin{align*}
& J_0(\t,\v)  = 
\cot\frac{\t}2  \ct \svp \int_0^1 t \sht \iint 
\frac{u \, \piu \, \pib}
{ (\cht - 1 + \q)^{\a + \b + 3} } \, dt, \\
&  
\ct \cvp  \int_0^1 t \sht \iint 
\frac{v \, \piu \, \pib}
{ (\cht - 1 + \q)^{\a + \b + 3} } \, dt \equiv J_{-1}(\t,\v) + J_1(\t,\v),
\end{align*}
where $J_{-1}(\t,\v)$ and $J_1(\t,\v)$ correspond to the integration in $v$ restricted to $[-1,0]$ and $[0,1]$, respectively. 
We will show that the operators associated with $J_0(\t,\v)$ and $J_{-1}(\t,\v)$ are bounded from $\Leb1$ to $\Lebr1$, whereas the one connected with $J_1(\t,\v)$ is unbounded. 

By \eqref{est6} the required property for $J_{-1}(\t,\v)$ is straightforward.
Next we focus on $J_0(\t,\v)$.
Combining Lemma~\ref{lem:uintcanc}~(a) (specified to $\gamma = \a + \b + 3$) with Lemma~\ref{lem:intest2} (applied with $\nu = 5/2$, $\gamma = 2(\a + \b + 3)$, $a=\sqrt{\q}$, $B = 1$) and 
Lemma~\ref{lem:K=0pi4} (with $\xi_1 = -\a - 1/2$, 
$\kappa_1 = \xi_2 = \kappa_2 = 0$, $\xi = \a + 3/2 > 1/2$) we infer that
\begin{align*}
|J_0(\t,\v)| & \lesssim 
\frac{\v^2}{(\t + \v)^2}  \iint 
\frac{ \piu \, \pib}
{ \q^{\a + \b + 3/2} } 
 \lesssim
\frac{\v^2}{(\t + \v)^2}
\frac{1}{|\t - \v|^{2\a + 2}}, \quad \, 
\t \in (0,\pi/4), \quad \v \in (0,\pi), \quad \t \ne \v.
\end{align*}
Now the conclusion for $J_0(\t,\v)$ is a direct consequence of Lemma~\ref{lem:Schur} and Lemma~\ref{lem:**} (specified to $\gamma = 2\a + 1$, $\lambda = 2$, $\nu = 0$ and $\kappa = 2\a + 2$).

Finally, we deal with $J_{1}(\t,\v)$. Since the integrand in the definition of $J_{1}(\t,\v)$ is non-negative, in view of Lemma~\ref{lem:Schur} it is sufficient to show that
\begin{align}\label{iden3}
\essup_{\v \in (0,\pi/8)} \int_0^{\pi/4} J_1(\t,\v) \, \t^{2\a+ 1} \, d\t
 = \infty.
\end{align}
Restricting the integration in $v$ to the interval $[1/2,1]$, using Lemma~\ref{lem:intest2} (specified to $\nu = 5/2$, $\gamma = 2(\a + \b + 3)$, $a=\sqrt{\q}$, $B = 1$) and then integrating in $u$, we see that
for $\t,\v \in (0,\pi/4)$, $\t \ne \v$, one has
\begin{align*}
J_1(\t,\v)
& \gtrsim
\int_0^1 t^2 \int_{[1/2,1]}  \int
\frac{ \piu \, \pib }{(t^2 + \q)^{\a+\b+3}} \, dt 
\simeq 
\int_{[1/2,1]}  \int
\frac{ \piu \, \pib }{ \q^{\a+\b+3/2}} \\
& \simeq 
\int_{[1/2,1]}  
\frac{  \pib }{ (1 - \st \svp - v \ct \cvp)^{\a+\b+3/2}}
\simeq 
\int  
\frac{  \pib }{ (1 - \st \svp - v \ct \cvp)^{\a+\b+3/2}};
\end{align*}
here we also used the fact that the essential contribution to the last integral comes from integration over $[1/2,1]$. Applying now Corollary~\ref{cor:intest3}
(with $\nu = \b$,  
$\gamma = \a + \b + 3/2$, $A = 1 - \st \svp \simeq 1$, $B=\ct \cvp$) we get
\[
\int_0^{\pi/4} J_1(\t,\v) \, \t^{2\a+ 1} \, d\t
\gtrsim
\int_0^{\pi/4} \frac{\t^{2\a+ 1} \, d\t}{|\t - \v|^{2\a + 2}} 
\gtrsim
\int_{2\v}^{\pi/4} \t^{-1} \, d\t 
=
\log \frac{\pi}{8 \v}, \qquad \v \in (0,\pi/8), 
\]
which confirms \eqref{iden3} and completes the case $-1 < \a < -1/2 \le \b$.

\noindent \textbf{Case 3:} $-1 < \b < -1/2 \le \a$.
From \cite[Proposition 2.3 (iii)]{NSS} we get
\begin{align*}
        \partial_{\t} {H}_t^{\ab}(\t,\v)
            = &  \; C_{\ab}^1 \sht \iint \frac{\partial_\t \q \, \partial_v \q}{(\cht - 1 + \q)^{\a+\b+4}} \, \pia \, \ob \\ 
            & + C_{\ab}^2 \sht \iint \frac{\partial_\t \partial_v \q }{(\cht - 1 + \q)^{\a+\b+3}} \, \pia \, \ob \\ 
            & + C_{\ab}^3 \sht \iint \frac{\partial_\t \q }{(\cht - 1 + \q)^{\a+\b+3}} \, \pia \, \piv, 
\end{align*}
where $C_{\ab}^j \ne 0$, $j=1,2,3$.
We denote by $L_j(\t,\v)$, $j=1,2,3$, the corresponding components of $K_0(\t,\v)$. 

We first show that the operators emerging from $L_2(\t,\v)$ and $L_3(\t,\v)$ are bounded from $\Leb1$ to $\Lebr1$. 
Indeed, applying  Lemma~\ref{lem:uintcanc} (a) (to the component of $L_3(\t,\v)$ connected with the first term in the decomposition \eqref{iden4} of $\partial_{\t} \q$) and also \eqref{est5} together with \eqref{smallab}, we obtain
\begin{align*}
|L_2(\t,\v)| + |L_3(\t,\v)|
& \lesssim \int_0^1 t^2
\iint 
\frac{\pia \, \pibb }{(t^2 + \q)^{\a+\b+3}}
\, dt  
+  \bigg( \frac{\v^2}{(\t + \v)^2} + 1 \bigg)\int_0^1 t^2
\iint 
\frac{\pia \, \piv }{(t^2 + \q)^{\a+\b+3}}
\, dt \\
& \simeq
\sum_{R=0,1} \int_0^1 t^2
\iint 
\frac{\pia \, \piR }{(t^2 + \q)^{\a+\b+3}}
 \, dt.
\end{align*}
Making use of Lemma~\ref{lem:intest2} 
(specified to $\nu = 5/2$, $\gamma = 2(\a + \b + 3)$, $a=\sqrt{\q}$, $B = 1$) and then Lemma~\ref{lem:K=0pi4}
(with $\xi_1 = \kappa_1 = \kappa_2 = 0$ and  
$ \xi_2 = - \b -1/2$, $\xi = \b + 3/2 > 1/2$ if $R=0$, 
and $ \xi_2 = 1$, $\xi = 0$ if $R=1$) yields
\begin{align*}
|L_2(\t,\v)| + |L_3(\t,\v)|
 \lesssim 
\sum_{R=0,1} 
\iint 
\frac{\pia \, \piR }{\q^{\a+\b+3/2}}
 \lesssim 
\frac{1}{(\t + \v )^{2\a + 1}} \, \frac{1}{|\t - \v|^{2\b + 2}},
\end{align*}
for $\t \in (0,\pi/4)$, $\v \in (0,\pi)$, $\t \ne \v$.
This, in view of Lemma~\ref{lem:Schur} and Lemma~\ref{lem:**} 
(specified to $\gamma = \lambda = - \nu = 2\a + 1$ and $\kappa = 2\b + 2$), gives the asserted property for $L_2(\t,\v)$ and $L_3(\t,\v)$.

It remains to investigate $L_1(\t,\v)$ and prove that it defines an unbounded operator from $L^1(d\mu_{\ab})$ to $L^1((0,\pi/4),d\mu_{\ab})$. 
Observe that, up to a non-zero multiplicative constant, $L_1(\t,\v)$ is equal to, see \eqref{iden1},
\begin{align*}
& \cot\frac{\t}2 \ct \cvp \int_0^1 
t \sht \iint 
\frac{\big[ \sin\frac{\t-\v}2  + (1-u) \ct \svp \big] - (1-v) \st \cvp}
{ (\cht - 1 + \q)^{\a + \b + 4} } \, \pia \, \ob \, dt. 
\end{align*}
This expression splits into two terms according to the main difference in the numerator of the fraction
under the double integral.
We denote by $J_{-1}(\t,\v) $ and $ J_1(\t,\v)$ the first of these terms with the integration in $v$ restricted to $[-1,0]$ and $[0,1]$, respectively. Further, let $J_2(\t,\v)$ stand for the second term.
We will prove that $\chi_{ \{ \t \le 2\v  \} } ( J_{-1}(\t,\v) + J_1(\t,\v) )$,
$ \chi_{ \{ \t > 2\v  \} } J_{-1}(\t,\v) $ and $J_2(\t,\v)$ define bounded operators from $\Leb1$ to $\Lebr1$, whereas 
$\chi_{ \{ \t > 2\v  \} } J_1(\t,\v)$
corresponds to an unbounded operator between those spaces.

Using (b) of Lemma~\ref{lem:-1/2} and \eqref{est6}, respectively, the asserted property for the first two kernels follows.
We now focus on $J_2(\t,\v)$. Splitting the integration in $v$ into intervals $[-1,0]$, $[0,1]$, and then using \eqref{est6} to the first term and the estimates \eqref{smallab}, \eqref{largeab} to the second one, we obtain
\begin{align*}
|J_2(\t,\v)| 
& \lesssim 
1 + 
\int_0^1 t^2
\iint 
\frac{ \pia \, d\Pi_{\b + 2}(v) }{(t^2 + \q)^{\a+\b+4}}
 \, dt.
\end{align*}
This together with Lemma~\ref{lem:intest2}
(specified to $\nu = 5/2$, $\gamma = 2(\a + \b + 4)$, $a=\sqrt{\q}$, $B = 1$) and Lemma~\ref{lem:K=0pi4}
(choosing $\xi_1 = \kappa_1 = \kappa_2 = 0$,  
$ \xi_2 = 2$ and $\xi = 0$) leads to
\[
|J_2(\t,\v)| \lesssim
(\t + \v)^{-2 \a} + |\t - \v|^{-1/2}
\lesssim
(\t + \v)^{-2 \a - 1}  |\t - \v|^{-1/2},
\qquad \t \in (0,\pi/4), \quad \v \in (0,\pi), \quad \t \ne \v.
\]
Now the conclusion follows from Lemma~\ref{lem:Schur} and Lemma~\ref{lem:**} 
(with $\gamma = \lambda = - \nu = 2\a + 1$ and $\kappa = 1/2$). 

Finally, we consider $\chi_{ \{ \t > 2\v \} } J_1(\t,\v)$.
Since the integrand in the definition of $\chi_{ \{ \t > 2\v \} } J_1(\t,\v)$ is non-positive, in view of Lemma~\ref{lem:Schur} it is enough to ensure that
\begin{align}\label{iden5}
\essup_{\v \in (0,\pi/8)} \int_{2\v}^{\pi/4} |J_1(\t,\v)| \, \t^{2\a+ 1} \, d\t
 = \infty.
\end{align}
Restricting the integration in $v$ to $[1/2, 1]$ and taking into account \eqref{smallab} we obtain
\[
\chi_{ \{ \pi/4 \ge \t > 2\v \} } | J_1(\t,\v) | 
\gtrsim
\chi_{ \{ \pi/4 \ge \t > 2\v \} } \, \t^{-1}
\int_0^1 t^2 \int_{[1/2,1]} \int
 \frac{ \t - \v }{ (t^2 + \q)^{\a + \b+ 4}} 
\, \pia \, \pibb \, dt .
\]
Since the last expression is comparable with a similar one with no restriction in $v$, an application of
Lemma~\ref{lem:intest2} (specified to $\nu = 5/2$, $\gamma = 2( \a + \b + 4)$, $a=\sqrt{\q}$, $B = 1$) and Corollary~\ref{cor:intest3} twice (first to the integral with respect to $\pibb$ with parameters $\nu = \b + 1$, $\gamma = \a + \b + 5/2$, $A = 1 - u \st \svp \simeq 1$, $B=\ct \cvp$ and then to the resulting integral with respect to $\pia$ with $\nu = \a$, $\gamma = \a + 1$, $A = 1 - \ct \cvp \simeq (\t + \v)^2$, $B=\st \svp$) gives
\begin{align*}
\chi_{ \{ \pi/4 \ge \t > 2\v \} } | J_1(\t,\v) | 
& \gtrsim
\chi_{ \{ \pi/4 \ge \t > 2\v \} } \,
\iint \frac{\pibb \, \pia}{  \q^{\a + \b+ 5/2}}  \\
& \simeq 
\chi_{ \{ \pi/4 \ge \t > 2\v \} } \, 
\int \frac{\pia}{  (1 - \ct \cvp - u \st \svp)^{\a + 1}}  
\simeq 
\chi_{ \{ \pi/4 \ge \t > 2\v \} } \, \t^{-2\a - 2},
\end{align*}
which implies \eqref{iden5}. This completes the reasoning justifying Lemma~\ref{lem:T1} for 
$-1 < \b < -1/2 \le \a$.

\noindent \textbf{Case 4:} $-1 < \ab < -1/2$.
We decompose $K_0(\t,\v)$ into $7$ parts according to the formula \eqref{qderiv} 
(note that the constants $C^j_{\ab}$, $j=1,\ldots,7$, appearing there are non-zero)
and denote them by $L_j(\t,\v)$, $j=1,\ldots,7$, respectively.
We first deal with $L_j(\t,\v)$, $j\ne 5,7$. We will show that these kernels produce bounded operators from $\Leb1$ to $\Lebr1$.
Using Lemma~\ref{lem:uintcanc} (b) (when it comes to $L_2(\t,\v)$ and $L_4(\t,\v)$) and \eqref{est5} and \eqref{smallab}, we see that they can be estimated simultaneously by
$$
	 \t^{-1} \int_0^1 t^2 \!\!
\sum_{\substack{K,R = 0,1 \\ K+R > 0 }} 
\iint \bigg(
\frac{ \chi_{ \{ K=1 \} } \, \t  \v \, \sqrt{\q} }{(t^2 + \q)^{\a+\b+4 + R}}
+
\frac{(\t  \v)^K}{ (\t + \v)^{2K} } 
 \frac{ \chi_{ \{ K=1 \} } \v + \chi_{ \{ R=1 \} } \t^{1+K} \v^K }
{(t^2 + \q)^{\a+\b+2 + K + R}}\,
\bigg) \piK \, \piR \, dt.
$$
Since 
$\t^K( \chi_{ \{ K=1 \} } \v + \chi_{ \{ R=1 \} } \t^{1+K} \v^K ) \lesssim \t \v^K$, $K=0,1$, an application of Lemma~\ref{lem:intest2} 
(with $\nu = 5/2$, $a=\sqrt{\q}$, $B = 1$ and $\gamma = 2(\a + \b + 4 + R)$ or $\gamma = 2(\a + \b + 2 + K + R)$) gives further bound by
\[
\frac{1}{|\t - \v|^{\a + 1}} + 
\sum_{\substack{K,R = 0,1 \\ K+R > 0 }} 
\iint \bigg(
\frac{ \chi_{ \{ K=1 \} } \, \v }{ \q^{\a+\b+2 + R}}
+
\frac{ \v^{2K} }{ (\t + \v)^{2K} } 
 \frac{ \chi_{ \{ \a + \b + 1/2 + K + R > 0 \} } }
{ \q^{\a+\b+ 1/2 + K + R} }
\bigg) \, \piK \, \piR;
\]
here we used also the relation $\q \gtrsim (\t - \v)^2$ and \eqref{logest} with $\rho = \a + 1$. Next, using Lemma~\ref{lem:K=0pi4} to the first term under the sum above
(specified to $\xi_1 = 1$, $\kappa_1 = \kappa_2 = 0$ and  
$ \xi_2 = 1$, $\xi = 1/2$ if $R=1$, and
$ \xi_2 = - \b -1/2$, $\xi = \b + 1 < 1/2$ if $R=0$) and also to the second term
(taken with $\xi_1 = \xi_2 = 1$, $\kappa_1 = \kappa_2 = \xi = 0$ if $K = R = 1$ and  $\xi_1 = \kappa_1 = 1/2$, $\xi_2 = - \b -1/2 $, $\kappa_2 = 0$, 
$\xi = \b + 1 < 1/2$ if $K=1$, $R=0$
and finally $\xi_1 = -\a - 1/2$, $\kappa_1 = 0$, $\xi_2 = \kappa_2 = 1/2$, 
$\xi = \a + 1 < 1/2$ if $K=0$, $R=1$), we infer that the expression in question is controlled by
\begin{align*}
& \frac{1}{|\t - \v|^{\a + 1}} 
+ 
\frac{ \v }{ (\t + \v)^{2\a + 3} } 
\bigg( \frac{ \t + \v }{ |\t - \v| } \bigg)^{1/2} 
+ 
\frac{ \v }{ (\t + \v)^{2\a + 2\b + 4} }
+
\frac{ \v^2  }{ (\t + \v)^{2\a + 4} } \\
& \qquad \lesssim
\frac{1}{|\t - \v|^{\a + 1}}
+
\frac{ \v }{ (\t + \v)^{2\a + 5/2} }
\frac{ 1 }{ |\t - \v|^{1/2} } 
, \qquad \t \in (0,\pi/4), \quad \v \in (0,\pi), \quad \t \ne \v ;
\end{align*}
here we used also \eqref{logest}.
This, in view of Lemma~\ref{lem:Schur} and Lemma~\ref{lem:**}
(applied with $\gamma = 2\a + 1$ and $\lambda = \nu = 0$, $\kappa = \a + 1$ or $\lambda = 2\a + 5/2$, $\nu = -2\a - 3/2$, $\kappa = 1/2$),
gives the desired conclusion for $L_j(\t,\v)$, $j \ne 5,7$. 

Next we show that $L_7(\t,\v)$ also produces a bounded operator from $\Leb1$ to $\Lebr1$.
Taking into account \eqref{iden4} and applying Lemma~\ref{lem:uintcanc} (a) (to the component of $L_7(\t,\v)$ connected with the first term in \eqref{iden4}) we obtain
\[
|L_7(\t,\v)| \lesssim
 \bigg(
\frac{  \, \v^2 }{(\t + \v)^2} + 1 \bigg) \int_0^1 t^2
\iint 
\frac{\piu \, \piv }{(t^2 + \q)^{\a+\b+3 }} \, dt.
\]
Using now Lemma~\ref{lem:intest2} 
(with $\nu = 5/2$, $\gamma = 2(\a + \b + 3)$, $a=\sqrt{\q}$, $B = 1$) 
and then integrating in $u$ and $v$ we get 
\[
|L_7(\t,\v)| 
\lesssim
\frac{1}{|\t - \v|^{\a + 1}}
+
\chi_{ \{ \a + \b + 3/2  > 0 \} }
\frac{1}{|\t - \v|^{2\a + 2\b + 3}},
\qquad \t \in (0,\pi/4), \quad \v \in (0,\pi), \quad \t \ne \v ,
\]
which with the aid of Lemma~\ref{lem:Schur} and Lemma~\ref{lem:**} 
(specified to $\gamma = 2\a + 1$, $\lambda =  \nu = 0$ and $\kappa = \a + 1$ or $\kappa = 2\a + 2\b + 3$) leads to the desired property of $L_7(\t,\v)$.

Finally, we consider $L_5(\t,\v)$. We denote by $J(\t,\v)$ the component of $L_5(\t,\v)$ corresponding to the first term in \eqref{iden4} and by $J_{-1}(\t,\v)$ and $J_1(\t,\v)$ the remaining parts of $L_5(\t,\v)$ with integration in $v$ restricted to $[-1,0]$ and $[0,1]$, respectively.
We aim at showing that $J_{-1}(\t,\v)$ and $J(\t,\v)$, in contrast with $J_1(\t,\v)$, stand behind bounded operators from $\Leb1$ to $\Lebr1$.
The conclusion for $J_{-1}(\t,\v)$ is a straightforward consequence of \eqref{est6}. We pass to analyzing $J(\t,\v)$.
Using sequently Lemma~\ref{lem:uintcanc}~(a), \eqref{smallab}, Lemma~\ref{lem:intest2}
(specified to $\nu = 5/2$, $\gamma = 2(\a + \b + 4)$, $a=\sqrt{\q}$, $B = 1$) 
and Lemma~\ref{lem:K=0pi4}
(with $\xi_1 = -\a -1/2$, $ \xi_2 = 1$, $\kappa_1 = \kappa_2 = 0$, 
$\xi = \a + 3/2 > 1/2$) we get
\begin{align*}
|J(\t,\v)|
& \lesssim
\frac{ \v^2 }{(\t + \v)^2}
\int_0^1 t^2 \iint
\frac{\piu \, \pibb }{(t^2 + \q)^{\a+\b+ 4}} \, dt
\simeq
\frac{ \v^2 }{(\t + \v)^2}
\iint
\frac{\piu \, \pibb }{ \q^{\a+\b+ 5/2}} \\
& \lesssim
\frac{ \v^2 }{(\t + \v)^2}
\frac{ 1 }{|\t - \v|^{2\a + 2}},
\qquad \t \in (0,\pi/4), \quad \v \in (0,\pi), \quad \t \ne \v .
\end{align*}
Now Lemma~\ref{lem:Schur} and Lemma~\ref{lem:**} 
(with $\gamma = 2\a + 1$, $\lambda = 2$, $ \nu = 0$ and $\kappa = 2\a + 2$) come into play,
and the conclusion follows. 

It remains to check that $J_1(\t,\v)$ gives rise to an operator that is not bounded from 
$L^1(d\mu_{\ab})$ to $L^1((0,\pi/4),d\mu_{\ab})$.
Since the integrand related to $J_1(\t,\v)$ is non-positive, it suffices to prove that
\begin{align}\label{iden6}
\essup_{\v \in (0,\pi/8)} \int_{0}^{\pi/4} |J_1(\t,\v)| \, \t^{2\a+ 1} \, d\t
 = \infty.
\end{align}
Restricting the integration in $v$ to the interval $[1/2,1]$ and applying \eqref{smallab} we obtain
\[
| J_1(\t,\v) | 
 \gtrsim
\int_0^1 t^2 \int_{[1/2,1]} \int 
 \frac{ \piu \, \pibb  }{ (t^2 + \q)^{\a + \b+ 4}} \, dt,
\qquad \t,\v \in (0,\pi/4), \quad \t \ne \v.
\]
Observe that the last expression is comparable to a similar one with integration in $v$ over the whole interval $[-1,1]$.
Using this, Lemma~\ref{lem:intest2} 
(specified to $\nu = 5/2$, $\gamma = 2(\a + \b + 4)$, $a=\sqrt{\q}$, $B = 1$),
Corollary~\ref{cor:intest3} (applied to the integral against $\pibb$ with 
$\nu = \b + 1$, $\gamma = \a + \b + 5/2$, $A= 1 - u \st \svp \simeq 1$, 
$B = \ct \cvp$) and then integrating in $u$, we see that
\begin{align*}
| J_1(\t,\v) | 
\gtrsim
\iint \frac{ \pibb \, \piu  }{  \q^{\a + \b+ 5/2}}  
 \simeq  
\int \frac{ \piu  }{  (1 - \ct \cvp - u \st \svp)^{\a + 1}} 
 \simeq 
|\t - \v|^{-2\a - 2},
\end{align*}
for $\t \in (0,\pi/4)$, $\v \in (0,\pi/8)$, $\t \ne \v$.
Consequently,
\[
\int_0^{\pi/4} | J_1(\t,\v) | \, \t^{2\a + 1} \, d\t 
\gtrsim
\int_{2\v}^{\pi/4} \t^{-1} \, d\t = \log \frac{\pi}{8 \v}, 
\qquad \v \in (0,\pi/8),
\]
which implies \eqref{iden6}.
This completes showing the case of $-1 < \ab < -1/2$ in Lemma \ref{lem:T1}.

The proof of Lemma~\ref{lem:T1} is at last finished.
\end{proof}



\begin{thebibliography}{99}

\bibitem{AFS}
H{.} Aimar, L{.} Forzani, R{.} Scotto,
\emph{On Riesz transforms and maximal functions in the context of Gaussian harmonic analysis},
Trans{.} Amer{.} Math{.} Soc{.} 359 (2007), 2137--2154.

\bibitem{BFMR}
J{.}J{.} Betancor, J{.}C{.} Fari\~na, T{.} Mart\'{i}nez, L{.} Rodr\'{i}guez-Mesa,
\emph{Higher order Riesz transforms associated with Bessel operators},
Ark{.}\ Mat{.}\ 46 (2008), 219--250.

\bibitem{BFRS}
J{.}J{.} Betancor, J{.}C{.} Fari\~na, L{.} Rodr\'{i}guez-Mesa, A{.} Sanabria-Garc\'{i}a,
\emph{Higher order Riesz transforms for Laguerre expansions},
Illinois J{.} Math{.} 55 (2011), 27--68.

\bibitem{BFRT}
J{.}J{.} Betancor, J{.}C{.} Fari\~na, L{.} Rodr\'{i}guez-Mesa, R{.} Testoni,
\emph{Higher order Riesz transforms in the ultraspherical setting as principal value integral operators},
Integr{.} Equ{.} Oper{.} Theory 70 (2011), 511--539.

\bibitem{Bur1}
D{.} Buraczewski, T{.} Mart\'{\i}nez, J{.}L{.} Torrea, R{.} Urban,
\emph{On the Riesz transform associated with the ultraspherical polynomials},
J{.} Anal{.} Math{.} 98 (2006), 113--143.

\bibitem{FHS}
L{.} Forzani, E{.} Harboure, R{.} Scotto,
\emph{Weak type inequality for a family of singular integral operators related with the Gaussian measure},
Potential Anal{.} 31 (2009), 103--116.

\bibitem{FSS}
L{.} Forzani, E{.} Sasso, R{.} Scotto,
\emph{Weak-type inequalities for higher order Riesz-Laguerre transforms},
J{.} Funct{.} Anal{.} 256 (2009), 258--274.

\bibitem{FS}
L{.} Forzani, R{.} Scotto,
\emph{The higher order Riesz transform for Gaussian measure need not be of weak type $(1,1)$},
Studia Math{.} 131 (1998), 205--214.

\bibitem{GMST}
J{.} Garc\'{i}a-Cuerva, G{.} Mauceri, P{.} Sj\"ogren, J{.}L{.} Torrea,
\emph{Higher-order Riesz operators for the Ornstein-Uhlenbeck semigroup},
Potential Anal{.} 10 (1999), 379--407.

\bibitem{L}
B{.} Langowski,
\emph{Harmonic analysis operators related to symmetrized Jacobi expansions},
Acta Math{.} Hungar{.} 140 (2013), 248--292.

\bibitem{Li}
Z{.} Li,
\emph{Conjugate Jacobi series and conjugate functions},
J{.} Approx{.} Theory 86 (1996), 179--196.

\bibitem{MuS}
B{.} Muckenhoupt, E{.}M{.} Stein,
\emph{Classical expansions and their relation to conjugate harmonic functions},
Trans{.} Amer{.} Math{.} Soc{.} 118 (1965), 17--92.

\bibitem{NR}
A{.} Nowak, L{.} Roncal,
\emph{Potential operators associated with Jacobi and Fourier-Bessel expansions},
preprint 2012. \texttt{arXiv:1212.6342}

\bibitem{NoSj}
A{.} Nowak, P{.} Sj\"ogren,
\emph{Calder\'on-Zygmund operators related to Jacobi expansions},
J{.} Fourier Anal{.} Appl{.} 18 (2012), 717--749.

\bibitem{NoSj2}
A{.} Nowak, P{.} Sj\"{o}gren,
\emph{Sharp estimates of the Jacobi heat kernel},
Studia Math{.} 218 (2013), 219--244.

\bibitem{NSS}
A{.} Nowak, P{.} Sj\"ogren, T{.}Z{.} Szarek,
\emph{Analysis related to all admissible parameters in the Jacobi setting},
preprint 2012. \texttt{arXiv:1211.3270}

\bibitem{NoSt0}
A{.} Nowak, K{.} Stempak, 
\emph{$L^2$-theory of Riesz transforms for orthogonal expansions}, 
J{.} Fourier Anal{.} Appl{.} 12 (2006), 675--711.

\bibitem{NoSt}
A{.} Nowak, K{.} Stempak,
\emph{A symmetrized conjugacy scheme for orthogonal expansions},
Proc{.} Roy{.} Soc{.} Edinburgh Sect{.} A 143 (2013), 427--443.

\bibitem{P}
S{.} P\'erez,
\emph{The local part and the strong type for operators related to the Gaussian measure},
J{.} Geom{.} Anal{.} 11 (2001), 491--507.

\bibitem{PS}
S{.} P\'erez, F{.} Soria,
\emph{Operators associated with the Ornstein-Uhlenbeck semigroup},
J{.} London Math{.} Soc{.} 61 (2000), 857--871.

\bibitem{S}
P{.} Sj\"ogren,
\emph{Operators associated with the Hermite semigroup - a survey},
J{.} Fourier Anal{.} Appl{.} 3 (1997), 813--823.

\bibitem{topics}
E{.}M{.} Stein,
\emph{Topics in Harmonic Analysis Related to the Littlewood-Paley Theory},
Annals of Math{.} Studies, Vol{.} 63, Princeton Univ{.} Press, Princeton, NJ, 1970.

\bibitem{Sz}
G{.} Szeg\"o,
\emph{Orthogonal Polynomials},
4th edn{.} Amer{.} Math{.} Soc{.} Colloq{.} Publ{.} vol{.} 23. Am{.} Math{.} Soc{.}, Providence (1975).

\bibitem{U}
W{.} Urbina,
\emph{On singular integrals with respect to the gaussian measure},
Ann{.} Scuola Norm{.} Sup{.} Pisa Cl{.} Sci{.} 17 (1990), 531--567. 

\bibitem{Z}
A{.} Zygmund,
\emph{Trigonometric series}, 2nd ed{.} vol{.} I. Cambridge University Press, New York (1959).

\end{thebibliography}
\end{document}